\documentclass[12pt]{article}
\usepackage[numbers,sort&compress]{natbib}
\usepackage{enumerate}
\usepackage{amscd}
\usepackage{amsmath}
\usepackage{latexsym}
\usepackage{amsfonts}
\usepackage{setspace}
\usepackage{amssymb}
\usepackage{amsthm}
\usepackage{verbatim}
\usepackage{mathrsfs}
\usepackage{enumerate}
\usepackage[hypertexnames=false]{hyperref}

\theoremstyle{plain}
\theoremstyle{definition}\newtheorem{theorem}{Theorem}[section]
\theoremstyle{definition}\newtheorem{lemma}[theorem]{Lemma}
\theoremstyle{plain}
\theoremstyle{plain}
\theoremstyle{definition}\newtheorem{remark}{Remark}[section]
\usepackage{xcolor}

\newcommand{\san}{\mathrm{div}}
\newcommand{\mr}{\mathbb{R}}
\newcommand{\mt}{\mathbb{T}}
\newcommand{\mz}{\mathbb{Z}}

\newcommand{\Li}{\Lambda^\iota}
\newcommand{\dd}{\mathrm{~d}}
\newcommand{\w}{\widehat}
\newcommand{\tb}{\tilde{b} }
\newcommand{\bs}{\breve{\sigma}}
\newcommand{\bw}{\breve{w}}
\newcommand{\f}{\frac}
\allowdisplaybreaks

\numberwithin{equation}{section}
\begin{document}
	\title{Asymptotic stability for $n$-dimensional isentropic compressible MHD equations without magnetic diffusion}
	\author{Quansen Jiu\footnote{School of Mathematical Sciences, Capital Normal University, Beijing, 100048, P. R. China. Email: jiuqs@cnu.edu.cn}~~~\,\,\,\,Jitao Liu\footnote{Department\,\,of\,\,Mathematics,\,\,School\,\,of\,\,Mathematics, \,\,Statistics\,\,and\,\,Mechanics,\,\,Beijing\,\,University\\
of\,\,Technology,\,\,Beijing,\,\,100124,\,\,P.\,\,R.\,\,China.\,\,Emails:\,\,jtliu@bjut.edu.cn,\,\,jtliumath@qq.com}~~~\,\,\,\,Yaowei Xie\footnote{School of Mathematical Sciences, Capital Normal University, Beijing, 100048, P. R. China. Email: mathxyw@163.com}}
\date{}
\maketitle
\begin{abstract}
Whether the global well-posedness of strong solutions of $n$-dimensional compressible isentropic magnetohydrodynamic ({\it MHD for short}) equations without magnetic diffusion holds true or not
remains an challenging open problem, even for the small initial data. In recent years, stared from the pioneer work by Wu and Wu {\it[Adv.\,\,Math.\,\,310\,\,(2017),\,\,759--888]}, much more attention has been paid to the system when the magnetic field near an equilibrium state ({\it the background magnetic field for short}). In particular, when the background magnetic field satisfies the Diophantine condition (see \eqref{Diophantine} for details), Wu and Zhai {\it [Math.\,\,Models\,\,Methods\,\,Appl.\,\,Sci.\,\,33 (2023), no.\,13, 2629--2656]} established the decay estimates and asymptotic stability for smooth solutions of the 3D compressible isentropic MHD system without magnetic diffusion in $H^{4r+7}(\mathbb{T}^3)$ with $r>2$ by exploiting a wave structure. In this paper, a new dissipative mechanism is found out and applied so that we can improve the spaces where the decay estimates and asymptotic stability of solutions are taking place
by Wu and Zhai. More precisely, we establish the decay estimates of solutions in $H^{r+1}(\mathbb{T}^n)$ and asymptotic stability result in  $H^{\left(3r+3\right)^+}(\mathbb{T}^n)$ for any dimensional periodic domain $\mathbb{T}^n$ with $n\geq 2$ and $r>n-1$. Our results  provide an approach for establishing the decay estimates and asymptotic stability in the Sobolev spaces with much lower regularity and uniform dimension, which can be used to study many other related models such as the compressible non-isentropic MHD system without magnetic diffusion and so on.
\end{abstract}
    decay estimates.
	
\section{Introduction}
\subsection{Background and Motivation}

\,\,\,\,\,\,\,\,In this paper, we are concerned with the following isentropic compressible MHD equations without magnetic diffusion in $n$-dimensional periodic domain $\mt^n=[-1,1]^n$ with $n\geq 2$,
	\begin{align}\label{original}
		\left\{\begin{array}{l}
			\rho_t+\nabla \cdot(\rho u)=0, \quad(t, x) \in \mathbb{R}^{+} \times \mathbb{T}^n, \\
			\rho u_t+\rho u \cdot \nabla u-\mu \Delta u-\lambda \nabla \nabla \cdot u+\nabla P=b\cdot\nabla b+\nabla( b \cdot  b), \\
			b_t+u \cdot \nabla b-b \cdot \nabla u=-b \nabla \cdot u, \\
			\nabla \cdot b=0, \\
			\rho(0, x)=\rho_0(x), \,\, u(0, x)=u_0(x), \,\, b(0, x)=b_0(x),
		\end{array}\right.
	\end{align}
where $\rho=\rho(x,t)$, $u=u(x,t)$ and $b=b(x,t)$ denote the density, the velocity field and the magnetic field respectively, $\lambda$ and $\mu$ are viscous coefficients satisfying
\begin{align*}
	\lambda>0, \quad \lambda+\mu>0.
\end{align*}
$P=P(\rho)$ denotes the scalar pressure, which is a strictly increasing and smooth function of $\rho$ satisfying $P'(\bar{\rho})=1$ for some constant reference density $\bar{\rho}>0$. Here,
\begin{align}\label{rho0}
	\bar{\rho}=\int_{\mathbb{T}^n}\rho(0,x) \dd x.
\end{align}

The MHD equations have been instrumental in accurately modeling significant phenomena in geophysics, astrophysics, cosmology, and engineering, including magnetic reconnection in astrophysics and geomagnetic dynamo in geophysics (see e.g., \cite{introduction1,introduction2}). Mathematically speaking, the compressible MHD equations with both viscosity and magnetic diffusion have been extensively studied, and important progress has been made in establishing its well-posedness problem, see \cite{wang-arma-2010-compressible-full,huang-2013-cmp-compresssible-full,suen-arma-2012-full,hong-2017-siam-full,liyachun-2023-jde-full} and the references therein. In particular, for
compressible isentropic MHD equations without magnetic diffusion \eqref{original}, the issue on global existence and stability of strong solutions near a background magnetic field has been attracting a lot of attention of mathematicians, and there have been rich research results in this area as well (see e.g., \cite{wu-2017-adv-compressible-nonresistive,tan-2018-siam-compressible-nonresistive,dong-2023-jns-nonresistive,jiang-2017-nonlinear-compressible-nonresistive,jiang-2023-arma-nonresistive,zhu-2022-jfa-nonresistive-nonresistive,zhai-3m-compressible-nonresistive}).

When the density $\rho$ is a constant, the system \eqref{original} becomes the incompressible MHD equations without magnetic diffusion, and extensive research has been made on the global
existence and stability of its solutions near a strong background magnetic field ${e_n}=(0,\cdots,0,1)$ or ${e_1}=(1,0,\cdots,0)$. Relevant results for the incompressible flow, please refer to \cite{abidi-2017-GlobalSolution3D,ren2014global,zhangting,panGlobalClassicalSolutions2018} and for the compressible flow, see \cite{wu-2017-adv-compressible-nonresistive,zhu-2022-jfa-nonresistive-nonresistive} for example. Specifically, for compressible isentropic MHD equations without magnetic diffusion and the pressure obeys the polytropic law, Wu and Wu \cite{wu-2017-adv-compressible-nonresistive} established its stability result near a strong background magnetic field in two-dimensional (2D for short) whole space $\mathbb{R}^2$ by transforming the original system \eqref{original} into a wave system. After this work, for the same perturbation system and general pressure, Wu and Zhu \cite{zhu-2022-jfa-nonresistive-nonresistive} obtained its stability result for 2D periodic domain $\mathbb{T}^2$.
However, whether or not strong solutions to the 3D system without magnetic diffusion near a strong background magnetic field are global well-posedness and stable in the whole space $\mathbb{R}^3$ or the periodic domain $\mathbb{T}^3$ remains an challenging open problem.

Except for the strong background magnetic field, there is another kind of background magnetic field which has attracted a lot of attention of mathematicians. This is the background magnetic field satisfies the following Diophantine condition.
{\bf Diophantine condition}: For every non zero vector $k \in \mathbb{Z}^n$ and $r>n-1$, there exists a constant $c>0$ such that
\begin{align}\label{Diophantine}
	|\tilde{b}\cdot k|\geq \frac{c}{|k|^r}.
\end{align}
It should be noted that according to the references \cite{Alinhac} and \cite{chen-2022-3dmhd-Diophant}, the Diophantine condition encompasses almost all vectors in $\mathbb{R}^n$.
As a matter of fact, when the magnetic field $b$ close to $\tb$, by setting $B=b - \tilde{b}$,  the system \eqref{original} becomes the following perturbation
\begin{align}\label{equation0}
	\left\{\begin{array}{l}
		\rho_t+\nabla \cdot(\rho u)=0, \quad(t, x) \in \mathbb{R}^{+} \times \mathbb{T}^n, \\
		\rho u_t+\rho u \cdot \nabla u-\mu \Delta u-\lambda \nabla \nabla \cdot u+\nabla P=h_1, \\
		B_t+u \cdot \nabla B-B \cdot \nabla u=h_2, \\
		\nabla \cdot B=0, \\
		\rho(0, x)=\rho_0(x),\,\, u(0, x)=u_0(x), \,\, B(0, x)=B_0(x),
	\end{array}\right.
\end{align}
where
\begin{equation}\label{h1}
h_1=\tb\cdot\nabla B+B\cdot\nabla B-\nabla(\tb\cdot B)-\nabla(B\cdot B),
\end{equation}
and
\begin{equation}\label{h2}
h_2=\tb\cdot\nabla u-B \nabla \cdot u-\tb \nabla\cdot u.
\end{equation}

For the incompressible system without either viscosity or magnetic diffusion near any background magnetic field satisfying the Diophantine condition, Chen, Zhang and Zhou \cite{chen-2022-3dmhd-Diophant} first studied its stability problem in the 3D periodic domain $\mathbb{T}^3$ and proved that its solution is asymptotic stable in the Sobolev spaces $H^{4r+7}(\mathbb{T}^3)$, which was then improved to  $H^{(3+2\beta)r+5+(\alpha+2\beta)}(\mathbb{T}^2)$ for 2D periodic domain $\mathbb{T}^2$ and any $\alpha>0$, $\beta>0$ by Zhai \cite{zhai-incompressible-2023-jde}. Later after \cite{chen-2022-3dmhd-Diophant} and \cite{zhai-incompressible-2023-jde}, by exploiting a new dissipative mechanism, we established the stability result in the Sobolev spaces with much lower regularity $H^{(3r+3)^+}(\mathbb{T}^n)$ for uniform dimension $n$ in \cite{preprite}, which improves the works of \cite{chen-2022-3dmhd-Diophant} and \cite{zhai-incompressible-2023-jde}. To better understand its stability theory, we also obtained the first linear stability result and discovered that the decay rates of linear and nonlinear systems are exactly same in \cite{preprite}.
Meanwhile, we found another interesting work by Jiang and Jiang \cite{jiang-2023-arma-nonresistive} for the 3D model without magnetic diffusion. For the background magnetic field satisfying the Diophantine condition and by setting $r=3$, they proved the stability of solutions for $(u_0,B_0)\in H^{17}(\mathbb{T}^3)\times H^{21}(\mathbb{T}^3)$ with some classes of large initial perturbation of magnetic field in Lagrangian coordinates.

As mentioned above, when the background magnetic filed satisfies the Diophantine condition, the research results on the stability of solutions of the incompressible flow have been very rich. However, the related study for the compressible flow is quite limited.
The initial work on this issue came from Wu and Zhai \cite{zhai-3m-compressible-nonresistive}, in which they established the stability result for the system \eqref{equation0}-\eqref{h2} when the density $\rho$ is close to $\bar{\rho}$ and the background magnetic field satisfies the Diophantine condition in 3D periodic domain $\mathbb{T}^3$. Following is a precise statement of their result.
\newline{\indent{\bf Theorem [Wu-Zhai,\cite{zhai-3m-compressible-nonresistive}].}} {\it Assume $\tilde{b}$ satisfies the Diophantine condition \eqref{Diophantine}. Let $m\geq 4r+7$ with $r>2$. Consider the system \eqref{equation0}-\eqref{h2} with the initial data $\left(\rho_0,u_0,B_0\right)$ satisfying
$$\rho_0-\bar{\rho}\in H^m(\mathbb{T}^3),\,\,\,c_0\leq \rho_0\leq c_0^{-1},\,\,\,u_0\in H^m(\mathbb{T}^3)\,\,\,B_0\in H^m(\mathbb{T}^3)$$
for some constant $c_0>0$. We further assume
$$\int_{\mt^n} \rho_0u_0\dd x=\int_{\mt^n}B_0\dd x=0.$$
Then there exists a small constant $\varepsilon$ such that, if
$$\left\|\rho_0-\bar{\rho}\right\|_{H^m}+\left\|u_0\right\|_{H^m}+\left\|B_0\right\|_{H^m} \leq \varepsilon,$$
then the system \eqref{equation0}-\eqref{h2} admits a unique global solution $(\rho-\bar{\rho},u,B)\in C([0,\infty);$
$H^m)$. Moreover, for any $t\geq0$ and $r+4\leq\beta\leq N$, there holds}
\begin{equation}\label{decaywz}
\|\left(\rho-\bar{\rho}\right)(t)\|_{H^\beta}+\|u(t)\|_{H^\beta}+\|B(t)\|_{H^\beta}\leq C(1+t)^{-\frac{3(m-\beta)}{2(m-r-4)}}.
\end{equation}
Later on, Li, Xu and Zhai \cite{lizhai-2023-jga} adapted the similar idea of \cite{zhai-3m-compressible-nonresistive} to prove the asymptotic stability of solutions of compressible non-isentropic MHD equations without magnetic diffusion in $H^{4r+7}(\mathbb{T}^3)$ with $r>2$.

Motivated by the work \cite{zhai-3m-compressible-nonresistive}, we would like to know if the solution could decay and become stable in the Sobolev spaces with {\it much lower} regularity and uniform dimension $n\ge 2$\,? In this paper, we give a positive answer to this question for the periodic domain $\mathbb{T}^n$.

\subsection{Main result}
\,\,\,\,\,\,\,\,Now, we are in the position to state our main result as follows.

\begin{theorem}\label{thm}
	Let $n\geq 2$ and $m> 3r+3$ with $r>n-1$. Assume that the initial data $(\rho_0-\bar{\rho},u_0, B_0) \in H^m(\mt^n)$ satisfies
	\begin{align}\label{b0}
		\int_{\mt^n} \rho_0u_0\dd x=\int_{\mt^n}B_0\dd x=0
	\end{align}
and
\begin{align}\label{rhoinitial}
	c_0\leq \rho_0\leq c_0^{-1},
\end{align}
for some constant $c_0>0$. If there is a small constant $\varepsilon>0$ such that
\begin{align}\label{small}
	\left\|\rho_0-\bar{\rho}\right\|_{H^m(\mt^n)}+\left\|u_0\right\|_{H^m(\mt^n)}+\left\|B_0\right\|_{H^m(\mt^n)} \leq \varepsilon,
\end{align}
then there exists a unique global solution $(\rho-\bar{\rho},u,B)\in C([0,\infty); H^m(\mt^n))$ solving the system \eqref{equation0}-\eqref{h2}. In addition, for any $t\geq 0$ and $r+1\leq \alpha\leq m$, the following decay estimates holds
\begin{align}
		\|\left(\rho-\bar{\rho}\right)(t)\|_{H^\alpha(\mt^n)}+\|u(t)\|_{H^\alpha(\mt^n)}+\|B(t)\|_{H^\alpha(\mt^n)}\leq C(1+t)^{-\frac{m-\alpha}{2(r+1)}}.
\end{align}
\end{theorem}

\begin{remark}\label{rmk3}
Theorem \ref{thm} {\it greatly reduces} the regularity requirement on the initial data of \cite{zhai-3m-compressible-nonresistive} from $H^{4r+7}(\mathbb{T}^3)$
to $H^{3(r+1)^+}(\mathbb{T}^3)$ for $r>2$.
\end{remark}

\begin{remark}\label{rmk1}
Compared with \eqref{decaywz}, in the Sobolev spaces with {\it much lower} regularity $H^{r+1}(\mt^n)$, we obtain the {\it effective} decay estimate.
\end{remark}

\begin{remark}\label{rmk7}
As pointed out in the references \cite{Alinhac} and \cite{chen-2022-3dmhd-Diophant}, the vectors consist of {\it rational numbers} or {\it at least one zero} are not contained in the Diophantine condition. Thus, our results are not contained in \cite{zhu-2022-jfa-nonresistive-nonresistive} concerning the global well-posedness of solutions close to the strong background magnetic field $\tilde{b}=(0,1)\in\mr^2$ in 2D periodic domain $\mathbb{T}^2$.
\end{remark}

\begin{remark}\label{rmk5}
Because our results hold for uniform dimension $n\geq 2$, Theorem \ref{thm} presents the first asymptotic stability result for 2D compressible isentropic MHD equations without magnetic diffusion near any background magnetic field satisfying the Diophantine condition.
\end{remark}

\begin{remark}\label{rmk6}
Our results provide an approach for establishing the decay estimates and asymptotic stability in the Sobolev spaces with much lower regularity and uniform dimension, which can be used to study many other related models such as compressible non-isentropic MHD system without magnetic diffusion and so on.
\end{remark}

\begin{remark}\label{rmk8}
Theorem \ref{thm} together with \cite{zhai-3m-compressible-nonresistive} reveal that any background magnetic field satisfying the Diophantine condition can stabilize the compressible isentropic MHD flow without magnetic diffusion.
\end{remark}

\subsection{Challenges and methodology}

\,\,\,\,\,\,\,\,In this subsection, we will explain why we can {\it greatly reduce} the regularity of Sobolev space in which the decay estimates and stability result of solution holds.
For clarity, we will divide it into following five parts.

{\bf Proof scheme}: Firstly, given the initial data  $(\rho_0-\bar{\rho},u_0, B_0) \in H^m(\mt^n)$, according to the literature \cite{local 1984}, the local well-posedness of solutions to the system \eqref{equation0}-\eqref{h2} can be proven via a rather standard procedure. Thanks to this fact, there exists a time $T > 0$ such that $(\rho-\bar{\rho},u, B)\in C([0,T]; H^m(\mt^n))$ solves the system \eqref{equation0}-\eqref{h2} uniquely and satisfies
\begin{align}\label{rho}
	\dfrac{1}{2}c_0\leq \rho(x,t)\leq 2c_0^{-1},
\end{align}
for any $t\in[0,T]$. Then, the main work is focused on how to extend this local solution to be a global one. The common method is the bootstrapping argument, whose specific operations are as follows. At the very beginning, we take the ansatz such that
\begin{align}\label{deltadelta}
		\sup_{t\in[0,T]}\left(\|\left(\rho-\bar{\rho}\right)(t)\|_{H^m}+\|u(t)\|_{H^m}+\|B(t)\|_{H^m}\right)\leq \delta,
	\end{align}
where $0<\delta<1$ is a small parameter. Assume that the initial data is chosen small, that satisfies
\begin{align}\label{smallsamll}
	\left\|\rho_0-\bar{\rho}\right\|_{H^m}+\left\|u_0\right\|_{H^m}+\left\|B_0\right\|_{H^m} \leq \varepsilon,
\end{align}
for sufficiently small $\varepsilon<0$ and then we can prove
\begin{align*}
		\sup_{t\in[0,T]}\left(\|\left(\rho-\bar{\rho}\right)(t)\|_{H^m}+\|u(t)\|_{H^m}+\|B(t)\|_{H^m}\right)\leq \frac{\delta}{2}.
	\end{align*}
At this moment, the bootstrapping argument tells us that the local solution becomes a global one.

{\bf Dissipation mechanism}: It is noted that in the system \eqref{equation0}-\eqref{h2}, there isn't any dissipation or damping in the equations of density and magnetic field. It is crucial to make use of the  the dissipative effects of quantities $\tb\cdot \nabla B$ and $\left(\rho-\bar{\rho}\right)+\tb\cdot B$ due to the background magnetic field (see \cite{zhai-3m-compressible-nonresistive}).

To simplify the analysis, as in \cite{zhai-3m-compressible-nonresistive}, we can assume that $\bar{\rho}=1$ and denote $a=\rho-\bar{\rho}=\rho-1$.
In particular, the system \eqref{equation0}-\eqref{h2} can be converted into
\begin{align}\label{equation}
\left\{\begin{array}{l}
	\partial_t a+\nabla\cdot u=f_1, \\
	\partial_t u-\nabla \cdot(\bar{\mu}(a) \nabla  u)-\nabla(\bar{\lambda}(a) \nabla \cdot u)+\nabla a+\nabla(\tb \cdot  B)-\tb \cdot \nabla  B=f_2, \\
	\partial_t  B-\tb \cdot \nabla  u+\tb \nabla \cdot u=f_3, \\
	\nabla \cdot B=0,
\end{array}\right.
\end{align}
where
	\begin{align}\label{definition}
	\bar{\mu}(a) \stackrel{\text { def }}{=} \frac{\mu}{1+a}=\f{\mu}{\rho},\,\,\bar{\lambda}(a) \stackrel{\text { def }}{=} \frac{\lambda}{1+a}=\f{\lambda}{\rho},\,\,I(a) \stackrel{\text { def }}{=} \frac{a}{1+a},\,\, k(a) \stackrel{\text { def }}{=} \frac{P^{\prime}(1+a)}{1+a}-1,
\end{align}
and
	\begin{align}\label{nonlinearity}
		\begin{aligned}
			& f_1 \stackrel{\text { def }}{=}- u \cdot \nabla a-a \nabla \cdot u, \\
			& \begin{aligned}
				f_2 \stackrel{\text { def }}{=} & - u \cdot \nabla  u+ B \cdot \nabla  B+
				\nabla( B \cdot  B)+k(a) \nabla a+\mu(\nabla I(a)) \nabla  u \\
				& +\lambda(\nabla I(a)) \nabla \cdot u-I(a)\left(\tb \cdot \nabla  B+ B \cdot \nabla  B-\nabla(\tb \cdot B) -\nabla( B \cdot  B)\right),
			\end{aligned} \\
			& f_3 \stackrel{\text { def }}{=}- u \cdot \nabla  B+ B \cdot \nabla  u- B \nabla \cdot u .
		\end{aligned}
	\end{align}
Concerning \eqref{rho}, without loss of generality, we assume
\begin{align}\label{assumption}
	\sup_{t\in[0,T],x\in \mathbb{T}^n}\left|a(t,x)\right|\leq \dfrac{1}{2}.
\end{align}
Define
	\begin{align*}
	\mathbb{D} \stackrel{\text { def }}{=} \nabla \Delta^{-1}\san ~~\mathrm{and}~~\mathbb{P}\stackrel{\text { def }}{=}\mathbb{I}-\mathbb{D}=\mathbb{I}-\nabla \Delta^{-1}\san\,\,({\rm Leray\,\,projection\,\, operator}),
\end{align*}
and apply the operators $\mathbb{P}$ and $\mathbb{D}$ to both sides of equation \eqref{equation}$_1$ to yield
\begin{align}\label{leray}
	&	\partial_t \mathbb{P}  u-\mu \Delta \mathbb{P}  u=  \tb \cdot \nabla  B+\mathbb{P} f_4,\\\label{nabla}
	&	\partial_t \mathbb{D}  u-(\lambda+\mu) \Delta \mathbb{D}  u+\nabla a+\nabla(  \tb \cdot  B)=\mathbb{D} f_4.
\end{align}
Then, by introducing two quantities
\begin{align*}
	w \stackrel{\text { def }}{=} a+\tb\cdot B, \quad \sigma \stackrel{\text { def }}{=} \mathbb{D} u-\dfrac{1}{\lambda+\mu}\Delta^{-1}\nabla w,
\end{align*}
Wu and Zhai \cite{zhai-3m-compressible-nonresistive} derived the following system involving a wave structure
\begin{align*}
	\renewcommand{\arraystretch}{2}
	\left\{\begin{array}{l}
		\partial_t w+\dfrac{1}{\lambda+\mu}\left(|\tb|^2+1\right) w+\left(|\tb|^2+1\right) \operatorname{div} \sigma=F_1,\\
		\partial_t \sigma-(\lambda+\mu) \Delta \sigma-\dfrac{1}{\lambda+\mu}\left(|\tb|^2+1\right) \mathbb{D} u=F_2 ,
	\end{array}\right.
\end{align*}
where $f_4$ and $F_1$, $F_2$ are explicitly shown in Subsection 3.2 and 3.4 respectively.

{\bf Main approaches}:  Using the classical energy method, one can get
\begin{align}
	&	\frac{\dd}{\dd t}\left[\underbrace{A\|   (u,B,a,\bs,\bw)\|_{H^{\frac{s}{2}+1}}^2-\int_{\mathbb{T}^n}\left(\tb\cdot\nabla B\right)\cdot \Lambda^s \mathbb{P}u\dd x-\int_{\mathbb{T}^n}\left(\tb\cdot\nabla B\right)\cdot \mathbb{P}u\dd x}_{I}\right]\nonumber\\
	&+\underbrace{A\left[\frac{1}{2}c_0\lambda \|  \san u\|_{H^{\frac{s}{2}+1}}^2+\frac{7}{32}c_0\mu \|\nabla  u\|_{H^{\frac{s}{2}+1}}^2+(\lambda+\mu)\|\nabla  \bs\|_{H^{\frac{s}{2}+1}}^2+\dfrac{1}{2(\lambda+\mu)}\|   \bw\|_{H^{\frac{s}{2}+1}}^2\right]}_{II}\nonumber\\
	\leq&-c_1\left\|\tb\cdot \nabla B\right\|_{H^{\frac{s}{2}}}^2+N,\label{apriori}
\end{align}
by setting $\breve{w}=\sqrt{\dfrac{3\left(\lambda+\mu\right)c_0\lambda}{64C_1C_4}}w$ and $\breve{\sigma}=\sqrt{\dfrac{c_0\lambda}{64C_4}}\sigma$, where $N$ represents the nonlinear terms and its expression is shown in \eqref{6.0}.

In  \cite{zhai-3m-compressible-nonresistive}, by choosing $n=3$ and $s=2r+6$, Wu and Zhai take $\left\|\tb\cdot \nabla B\right\|_{H^{r+3}}$ as a dissipative term directly and use it together with term $II$ to control the nonlinear terms $N$ and term $I$. By setting
\begin{align*}
G(t)&\triangleq A\|   (u,B,a,\bs,\bw)\|_{H^{r+4}}^2-\int_{\mathbb{T}^3}\left(\tb\cdot\nabla B\right)\cdot \Lambda^{r+4} \mathbb{P}u\dd x-\int_{\mathbb{T}^3}\left(\tb\cdot\nabla B\right)\cdot \mathbb{P}u\dd x \\
&\geq\frac{A}{2} \|   (u,B,a,\bs,\bw)\|_{H^{r+4}}^2
\end{align*}
and
\begin{align*}
H(t)&\triangleq A\left[\frac{1}{2}c_0\lambda \|  \san u\|_{H^{r+4}}^2+\frac{7}{32}c_0\mu \|\nabla  u\|_{H^{r+4}}^2+(\lambda+\mu)\|\nabla  \bs\|_{H^{r+4}}^2+\dfrac{1}{2(\lambda+\mu)}\|   \bw\|_{H^{r+4}}^2\right]\\
&+c_1\left\|\tb\cdot \nabla B\right\|_{H^{r+3}}^2,
\end{align*}
they got the following dissipative inequality
\begin{equation*}\label{WZinequality}
 \frac{\dd}{\dd t}G(t)+c\left(H(t)\right)^{\frac43}\leq0,
\end{equation*}
which implies
\begin{equation}\label{WZdecay}
 \|   (u,B,a,\bs,\bw)\|_{H^{r+4}({\mathbb{T}^3})}\leq C(1+t)^{-\frac{3}{2}}.
\end{equation}
This together with \eqref{smallsamll} and the bootstrapping argument then leads to the stability result.

Our main contribution of this paper is to establish the decay rate of $(a,u,B)$ in the Sobolev spaces with {\it much lower} regularity by further exploiting potential dissipation effect in \eqref{apriori}. In comparison with \cite{zhai-3m-compressible-nonresistive}, we first use the coupling and interaction between the term $c_1\left\|\tb\cdot \nabla B\right\|_{H^{\frac{s}{2}}}^2$ and half of $\dfrac{1}{2(\lambda+\mu)}\|   \bw\|_{H^{\frac{s}{2}+1}}^2$, together with the Diophantine condition satisfied by background magnetic filed to recover the dissipative effect of $a$ and $B$. Following this idea and unifying the various constants, we update \eqref{apriori} as
\begin{align}\nonumber
	&	\frac{\dd}{\dd t}\left[A\|   (u,B,a,\bs,\bw)\|_{H^{\frac{s}{2}+1}}^2-\int_{\mathbb{T}^n}\left(\tb\cdot\nabla B\right)\cdot \Lambda^s \mathbb{P}u\dd x-\int_{\mathbb{T}^n}\left(\tb\cdot\nabla B\right)\cdot \mathbb{P}u\dd x\right]\\\label{apriori1}
	&~+\varkappa\left[\underbrace{\|B\|_{H^{\frac{s}{2}-r}}^2+\|a\|_{H^{\frac{s}{2}-r}}^2}_{III}+\|\nabla u\|_{H^{\frac{s}{2}+1}}^2+\|\nabla  \bs\|_{H^{\frac{s}{2}+1}}^2+\|   \bw\|_{H^{\frac{s}{2}+1}}^2\right]\leq N.
\end{align}
Although the dissipative effect of term $III$ in \eqref{apriori1} is not strong enough to produce any decay rate, we restore the dissipative effect from the term $III$ to succeed to control $N$. We can further update \eqref{apriori1} as
\begin{align}\nonumber
	&	\frac{\dd}{\dd t}\left[A\|   (u,B,a,\bs,\bw)\|_{H^{\frac{s}{2}+1}}^2-\int_{\mathbb{T}^n}\left(\tb\cdot\nabla B\right)\cdot \Lambda^s \mathbb{P}u\dd x-\int_{\mathbb{T}^n}\left(\tb\cdot\nabla B\right)\cdot \mathbb{P}u\dd x\right]\\\nonumber
	\leq& -\dfrac{\varsigma}{2MC_6}\left[A\|   (u,B,a,\bs,\bw)\|_{H^{\frac{s}{2}+1}}^2-\int_{\mathbb{T}^n}\left(\tb\cdot\nabla B\right)\cdot \Lambda^s \mathbb{P}u\dd x-\int_{\mathbb{T}^n}\left(\tb\cdot\nabla B\right)\cdot \mathbb{P}u\dd x\right]\\\label{apriori2}
	&+\underbrace{\frac{C_7}{{M}^{1+\frac{m-\frac{s}{2}-1}{1+r}}}\left(\|a\|_{H^m}^2+\|B\|_{H^m}^2\right)}_{V},
\end{align}
from which we get desired dissipative effect. The price to pay is the appearance of forcing term $V$. To deal with the forcing term $V$, we introduce a suitable Lyapunov functional
\begin{align*}
	F(t)&\triangleq A\|   (u,B,a,\bs,\bw)\|_{H^{\frac{s}{2}+1}}^2-\int_{\mathbb{T}^n}\left(\tb\cdot\nabla B\right)\cdot \Lambda^s \mathbb{P}u\dd x-\int_{\mathbb{T}^n}\left(\tb\cdot\nabla B\right)\cdot \mathbb{P}u\dd x\\
&\geq\frac{A}{2}\|   (u,B,a,\bs,\bw)\|_{H^{\frac{s}{2}+1}}^2,
\end{align*}
with the help of which, we succeed to establish the following decay rate (see Lemma \ref{lem4} for more details)
\begin{align}
	\|   (u,B,a,\bs,\bw)\|_{H^{\frac{s}{2}+1}}\leq C(1+t)^{-\frac{m-(\frac{s}{2}+1)}{2(1+r)}}.\label{ourdecay}
\end{align}
  It is remarked that in \eqref{ourdecay}, $\frac{s}{2}+1\geq r+1$ for any dimension $n$, while the corresponding index in \eqref{WZdecay} is $r+4$. This means that we establish the decay estimates in Sobolev spaces with {\it much lower} regularity, which enables us to get the stability result in Sobolev spaces with {\it much lower} regularity.

\subsection{Organization of the paper}
\,\,\,\,\,\,\,\,In Section 2, a multitude of beneficial lemmas will be presented.
In Section 3, we focus on establishing a priori estimates.
More precisely, in Subsection 3.1, we will establish several Poincaré-type inequalities for $a$, $u$ and $B$ as preparation.
In Section 3.2, classical energy estimates for $a$, $u$ and $B$ will be provided.
Section 3.3 reveals a hidden dissipation mechanism for B.
In Section 3.4, in order to obtain the dissipative mechanism for $a$, two new quantities will be introduced and their energy estimates will be established.
In Section 4, we will concentrate on proving Theorem \ref{thm}.
To be specific, in Subsection 4.1, the effective decay rate of solution will be established, which is the core part of this paper.
In Section 4.2, we will use the bootstrapping argument to finish all the proof.

\section{Preliminaries}

\,\,\,\,\,\,\,\,The first Lemma presents one Poincaré-type inequality involving the Diophantine condition.
\begin{lemma}\label{Diophantinepoincare}
	Assume $\tb\in\mr^n$ be given such that the Diophantine condition \eqref{Diophantine} holds. If $f\in H^{s+r+1}$ satisfying $\int_{\mathbb{T}^n}f\dd x=0$, then for any $s\in \mathbb{R}$, it holds that
		\begin{align}\label{2.2}
		\|f\|_{H^s}\leq C \|\tilde{b}\cdot\nabla f\|_{H^{s+r}}.
	\end{align}

	\begin{proof}
		According to Plancherel's theorem, it follows that
		\begin{align*}
			\|\tilde{b}\cdot\nabla f\|_{H^{s+r}}^2&=\sum_{k\in\mz^n}\left(1+|k|^2\right)^{s+r}|\tilde{b}\cdot k|^2|\w{f}(k)|^2\\
			&=\sum_{k\in\mz^n\backslash\{0\}}\left(1+|k|^2\right)^{s+r}|\tilde{b}\cdot k|^2|\w{f}(k)|^2\\
			&\geq c\sum_{k\in\mz^n\backslash\{0\}}\left(1+|k|^2\right)^{s+r}|k|^{-2r}|\w{f}(k)|^2\\
			&\geq c\sum_{k\in\mz^n\backslash\{0\}}\left(1+|k|^2\right)^{s}|\w{f}(k)|^2=c\|f\|_{H^s}^2.
\end{align*}
		which completes the proof.
	\end{proof}
\end{lemma}

Next, we introduce two weighted Poincaré-type inequalities when the density $\rho$ appears, that is quite different from the incompressible homogenous flow. The readers can refer to IV.2. in \cite{poincare1-invent} and Lemma 3.2 in \cite{poincare2-feireisl} for details.
\begin{lemma}\label{poincare1}
	Let $u\in H^1(\Omega)$, $\rho$ be a non-negative function and $\varLambda$ be a positive constant such that
	\begin{align*}
		\int_{\Omega}\rho(x,t)\dd x=1,\quad\rho\leq \varLambda,
	\end{align*}
where $\Omega\subset\mathbb{R}^n$ is a bounded domain. Then there exists a constant $C$ depending only on $\Omega$ and $\varLambda$ such that
\begin{align*}
	\int_{\Omega}\rho\left(u-\int_{\Omega}\rho u\dd x\right)^2\dd x\leq C(\Omega,\varLambda)\|\nabla u\|_{L^2}^2.
\end{align*}
\end{lemma}

\begin{lemma}\label{poincare2}
	Let $u\in H^1(\Omega)$, $\gamma>1$, $\rho$ be a non-negative function and given two positive constants $\varrho, \kappa$  such that
	\begin{align*}
		0< \varrho \leq \int_{\Omega}\rho(x,t)\dd x, \quad\int_{\Omega}\rho^\gamma(x,t)\dd x\leq \kappa,
	\end{align*}
where $\Omega\subset\mathbb{R}^n$ is a bounded domain. Then there holds
\begin{align*}
	\|u\|_{L^2}^2\leq C(\kappa,\varrho)\left[\|\nabla u\|_{L^2}^2+\left(\int_{\Omega}\rho |u|\dd x\right)^2\right],
\end{align*}
where $C>0$ is an absolute constant depending only on $\kappa,\varrho$.
\end{lemma}

\begin{lemma}\label{smooth function}
\cite{bahouri fourier}	Let $F$ be a smooth function vanishing at $0$, $s>0$ and $p,r\in[1,+\infty)$. If $f\in B^s_{p,r}(\mt^n)\cap L^\infty(\mt^n)$, then so does $F(f)$, and there exists a constant $C$ such that
	\begin{align*}
		\|F(f)\|_{B^s_{p,r}}\leq C\left(s,F',\|f\|_{L^\infty}\right)\|f\|_{B^s_{p,r}}.
	\end{align*}
Specifically, if $p=r=2$, we have $B^s_{2,2}(\mt^n)=H^s(\mt^n)$ and therefore
\begin{align*}
	\|F(f)\|_{H^s}\leq C\left(s,F',\|f\|_{L^\infty}\right)\|f\|_{H^s}.
\end{align*}
\end{lemma}

Recalling the definition of $I(a)$ and $k(a)$ in \eqref{definition}, the property of $P'(1+a)$ and \eqref{rho}, we know it is a smooth function satisfying $I(0)=0$, $K(0)=0$ and $a\in L^\infty$. As a result, according to Lemma \ref{smooth function} and some basic calculations, it holds that
\begin{align}\label{Ia}
	\|I(a)\|_{H^s}+\|K(a)\|_{H^s}\leq C\|a\|_{H^s},
\end{align}
and
\begin{align}\label{Ia1}
	\|\nabla I(a)\|_{L^\infty}\leq C\|\nabla a\|_{L^\infty},
\end{align}
where $C$ depends only on $s$ and $c_0$. The next lemma provides a commutator estimates that can be found in \cite{kato-1988-CommutatorEstimatesEuler,kenig}.

\begin{lemma}\label{commutator}
	Let $s>0,1<p<\infty$ and $\frac{1}{p}=\frac{1}{p_1}+\frac{1}{q_1}=\frac{1}{p_2}+\frac{1}{q_2}$.
	\begin{itemize}
		\item For any $f\in W^{1,p_1}\cap W^{s,q_2},g\in L^{p_2}\cap W^{s-1,q_1}$, there exists an absolute constant $C$ such that
		\begin{align*}
			\left\|\Lambda^s(f g)-f \Lambda^s g\right\|_{L^p} \leq C\left(\|\nabla f\|_{L^{p_1}}\left\|\Lambda^{s-1} g\right\|_{L^{q_1}}+\|g\|_{L^{p_2}}\left\|\Lambda^s f\right\|_{L^{q_2}}\right).
		\end{align*}
		\item  If $f\in L^{p_1}\cap W^{s,q_2},g\in L^{p_2}\cap W^{s,q_1}$, there is an absolute constant $C$ such that
		\begin{align*}
			\left\|\Lambda^s(f g)\right\|_{L^p} \leq C\left(\|f\|_{L^{p_1}}\left\|\Lambda^s g\right\|_{L^{q_1}}+\|g\|_{L^{p_2}}\left\|\Lambda^s f\right\|_{L^{q_2}}\right).
		\end{align*}
	\end{itemize}
\end{lemma}

Finally, we introduce one Gagliardo-Nirenberg interpolation inequality that comes from \cite{brezis-2019-gag}.

\begin{lemma}\label{chazhi1}
	If $f\in W^{s_1,p_1}\bigcap W^{s_2,p_2}$, then
	\begin{align*}
		\|f\|_{W^{r,q}}\leq 	\|f\|^{\theta}_{W^{s_1,p_1}}	\|f\|^{1-\theta}_{W^{s_2,p_2}},
	\end{align*}
where the real numbers $0 \leq s_1 \leq s_2,~ r \geq 0$, $1 \leq p_1, p_2, q \leq \infty$, $\left(s_1, p_1\right) \neq\left(s_2, p_2\right)$ and $\theta \in(0,1)$ such that
\begin{align*}
	& r<s:=\theta s_1+(1-\theta) s_2,~ \frac{1}{q}=\left(\frac{\theta}{p_1}+\frac{1-\theta}{p_2}\right)-\frac{s-r}{n}.
\end{align*}
\end{lemma}

\section{A priori estimates}

\subsection{The Poincaré-type inequalities of $a,B,u$}
\,\,\,\,\,\,\,\,As a preparation, we first sort out all available Poincaré-type inequalities. By recalling the initial assumptions \eqref{rho0} and \eqref{b0} and using the equations \eqref{equation}, for any $t\geq 0$, one has
\begin{align}\label{modezero-aB}
	\int_{\mt^n}a(x,t)\dd x=\int_{\mt^n}B(x,t)\dd x=0.
\end{align}
Then, by utilizing Plancherel's theorem, for any $0\leq s_1\leq s_2$, it yields that
\begin{align}\nonumber
	\|\Lambda^{s_1}B(x,t)\|_{L^2}^2&=\sum_{|k|\neq 0,k\in \mathbb{Z}^n}|k|^{2s_1}\left|\w{B}(k,t)\right|^2\\\label{bpoincare}
	&\leq \sum_{|k|\neq 0,k\in \mathbb{Z}^n}|k|^{2s_2}\left|\w{B}(k,t)\right|^2=	\|\Lambda^{s_2}B(x,t)\|_{L^2}^2
\end{align}
and
\begin{align}\label{apoincare}
	\|\Lambda^{s_1}a(x,t)\|_{L^2}^2\leq	\|\Lambda^{s_2}a(x,t)\|_{L^2}^2.
\end{align}

Regarding the velocity $u$, unfortunately, due to the appearance of density $\rho$, its property of zero average integral dose not hold any more, which is quite different from the incompressible homogenous flow \cite{preprite}.  In this case, if one assumes
\begin{align*}
	\int_{\mathbb{T}^n}\rho_0u_0\dd x=0,
\end{align*}
then by utilizing \eqref{equation}$_1$ and \eqref{equation}$_2$, it follows that
\begin{align*}
	\int_{\mathbb{T}^n}\rho(x,t) u(x,t)\dd x=0.
\end{align*}
Thus, by applying Lemma \ref{poincare1}, we have
\begin{align*}
	\|\sqrt{\rho}u\|_{L^2}^2=\int_{\mt^n}\rho\left(u-\int_{\mt^n}\rho u\dd x\right)^2\dd x\leq C\|\nabla u\|_{L^2}^2,
\end{align*}
which further implies, after combining \eqref{rho} and Lemma \ref{poincare2}, that
\begin{align}\nonumber
	\|u\|_{L^2}^2&\leq C\left[\|\nabla u\|_{L^2}^2+\left(\int_{\mt^n}\rho |u|\dd x\right)^2\right]\\\nonumber
	&\leq C\|\nabla u\|_{L^2}^2 +C\left(\int_{\mt^n}|\rho|\dd x\int_{\mt^n}|\rho||u|^2\dd x\right)\\\label{upoincare}
	&\leq C\|\nabla u\|_{L^2}.
\end{align}
\subsection{Energy estimates}

\begin{lemma}\label{hign energy}
Assume that $(a, u, B) \in C([0, T]; H^m)$ is a solution to the system \eqref{equation}. Then for any number $\iota =0$ or $1\leq \iota\leq m$, it holds that
\begin{align}\nonumber
	&\frac{1}{2}	\frac{\dd}{\dd t}\|    (a,u,B)\|_{  H^\iota}^2+\frac{1}{4}c_0\lambda \|  \san u\|_{  H^\iota}^2+\frac{3}{16}c_0\mu \|\nabla   u\|_{  H^\iota}^2 \\\label{3.00}
	\leq& C\|   (a,u,B)\|_{  H^\iota}^2\left(\| (a,u,B)\|_{W^{1,\infty}}+\|(a,u,B)\|_{W^{1,\infty}}^2+\|(a,B)\|_{W^{1,\infty}}^4\right).
\end{align}
\end{lemma}
\begin{proof}
{\bf When $\iota\geq 1$}, applying the operator $\Li$ on both sides of \eqref{equation}$_1$, \eqref{equation}$_2$ and \eqref{equation}$_3$, taking the inner product of resultants with $\Li a$, $\Li u$ and $\Li B$ respectively and adding them up, we have
\begin{align}\label{3.0}
		&\frac{1}{2}	\frac{\dd}{\dd t}\|\Li  (a,u,B)\|_{L^2}^2-\int_{\mt^n}\Li \san (\bar{\mu}(a) \nabla  u)\cdot\Li u\dd x-\int_{\mt^n}\Li \nabla\left( \bar{\lambda}(a) \nabla \cdot u\right)\cdot\Li u\dd x\notag\\
=&\sum_{j=1}^{9}I_j,
\end{align}
	where
	\begin{align*}
		I_1=&- \int_{\mathbb{T}^n} \Li  (\nabla\cdot  u) \cdot \Li a \mathrm{~d} x,
		~I_2= -\int_{\mathbb{T}^n}\Li (\nabla a) \cdot \Li  u \mathrm{~d} x, \\
		I_3=& - \int_{\mathbb{T}^n}\Li  \nabla(\tb \cdot  B) \cdot \Li  u\mathrm{~d} x,
		~I_4= \int_{\mathbb{T}^n} \Li  (\tb \cdot \nabla  B) \cdot \Li  u\mathrm{~d} x,\\
		I_5=& \int_{\mathbb{T}^n} \Li (\tb \cdot \nabla  u) \cdot \Li  B \mathrm{~d} x, ~I_6= -\int_{\mathbb{T}^n} \Li (\tb \nabla \cdot u)\cdot \Li  B\mathrm{~d} x,\\
		I_7=&\int_{\mathbb{T}^n} \Li  (f_1) \cdot \Li a \mathrm{~d} x,	~I_8=\int_{\mathbb{T}^n} \Li  (f_2) \cdot \Li u \mathrm{~d} x,	~I_9=\int_{\mathbb{T}^n} \Li  (f_3) \cdot \Li B \mathrm{~d} x.
	\end{align*}
	By some basic calculations and integration by parts, it holds
	\begin{align*}
		I_1+I_2&=- \int_{\mathbb{T}^n} \Li  (\nabla\cdot  u) \cdot \Li a \mathrm{~d} x -\int_{\mathbb{T}^n}\Li (\nabla a) \cdot \Li  u \mathrm{~d} x\\
		&=\int_{\mathbb{T}^n} \Li  u \cdot \nabla \Li a \mathrm{~d} x -\int_{\mathbb{T}^n}\Li (\nabla a) \cdot \Li  u \mathrm{~d} x=0,\\
		I_3+I_5&=- \int_{\mathbb{T}^n}\Li  \nabla(\tb \cdot  B) \cdot \Li  u\mathrm{~d} x+\int_{\mathbb{T}^n} \Li (\tb \cdot \nabla  u) \cdot \Li  B \mathrm{~d} x=0,\\
		I_4+I_6&=\int_{\mathbb{T}^n} \Li  (\tb \cdot \nabla  B) \cdot \Li  u\mathrm{~d} x-\int_{\mathbb{T}^n} \Li (\tb \nabla \cdot u)\cdot \Li  B\mathrm{~d} x=0,
	\end{align*}
and
	\begin{align}\nonumber
		&-\int_{\mt^n}\Li \san (\bar{\mu}(a) \nabla  u)\cdot\Li u\dd x=\int_{\mt^n}\Li  (\bar{\mu}(a) \nabla  u)\cdot\Li \nabla u\dd x\\
		=&\int_{\mt^n}\left[\Li  (\bar{\mu}(a) \nabla  u)-\bar{\mu}(a)\Li \nabla  u+\bar{\mu}(a)\Li \nabla  u\right]\cdot\Li \nabla u\dd x.\label{3.1}
	\end{align}
By virtue of \eqref{definition} and \eqref{rho}, it follows that
	\begin{align}
		\int_{\mt^n}\bar{\mu}(a)\Li \nabla  u\cdot\Li \nabla u\dd x\geq \frac{1}{2}c_0\mu \|\nabla \Lambda^{\iota} u\|_{L^2}^2.\label{3.2}
	\end{align}
Regarding the remaining terms of \eqref{3.1}, by making use of the following relation		
	\begin{align*}
		\bar{\mu}(a)-\mu=\mu\left(\dfrac{1}{1+a}-1\right)=-\mu\dfrac{a}{1+a}=-\mu I(a),
	\end{align*}
we have
	\begin{align*}
	&\int_{\mt^n}\left[\Li  (\bar{\mu}(a) \nabla  u)-\bar{\mu}(a)\Li \nabla  u\right]\cdot\Li \nabla u\dd x\\
=&\int_{\mt^n}\left[\Li,\bar{\mu}(a)\right]\nabla u\cdot\Li \nabla u\dd x\\
		=&\int_{\mt^n}\left[\Li,\bar{\mu}(a)-\mu\right]\nabla u\cdot\Li \nabla u\dd x\\
		=&-\int_{\mt^n}\left[\Li,\mu I(a)\right]\nabla u\cdot\Li \nabla u\dd x,
	\end{align*}
which yields, after employing Lemma \ref{commutator}, \eqref{Ia} and \eqref{Ia1}, that
\begin{align}\nonumber
&\left|	\int_{\mt^n}\left[\Li,\mu I(a)\right]\nabla u\cdot\Li \nabla u\dd x\right|\\\nonumber
\leq& C\|\nabla\Lambda^{\iota} u\|_{L^2}\left(\|\nabla I(a)\|_{L^\infty}\|\Li u\|_{L^2}+\|\nabla u\|_{L^\infty}\|\Li I(a)\|_{L^2}\right)\\\label{3.3}
\leq& \frac{1}{4}c_0\mu\|\nabla\Lambda^{\iota} u\|_{L^2}^2+C \left(\|\nabla a\|_{L^\infty}^2\|\Li u\|_{L^2}^2+\|\nabla u\|_{L^\infty}^2\|\Li a\|_{L^2}^2\right).
\end{align}
Substituting \eqref{3.2} and \eqref{3.3} into \eqref{3.1} leads to
\begin{align*}
		&-\int_{\mt^n}\Li \san (\bar{\mu}(a) \nabla  u)\cdot\Li u\dd x\\
		\geq& \frac{1}{4}c_0\mu \|\nabla\Lambda^{\iota} u\|_{L^2}^2-C \left(\|\nabla a\|_{L^\infty}^2\|\Li u\|_{L^2}^2+\|\nabla u\|_{L^\infty}^2\|\Li a\|_{L^2}^2\right).
\end{align*}

In a similar manner, one can estimate the third term on the left side of \eqref{3.0} that
\begin{align*}
	&-\int_{\mt^n}\Li \nabla \left( \bar{\lambda}(a) \nabla \cdot u\right)\cdot\Li u\dd x\\
	\geq& \frac{1}{4}c_0\lambda \|\Lambda^{\iota}\san u\|_{L^2}^2-C \left(\|\nabla a\|_{L^\infty}^2\|\Li u\|_{L^2}^2+\|\nabla u\|_{L^\infty}^2\|\Li a\|_{L^2}^2\right)
\end{align*}
Collecting all estimates above leads to
\begin{align}\nonumber
	&\frac{1}{2}	\frac{\dd}{\dd t}\|\Li  (a,u,B)\|_{L^2}^2+\frac{1}{4}c_0\lambda \|\Lambda^{\iota}\san u\|_{L^2}^2+\frac{1}{4}c_0\mu \|\nabla\Lambda^{\iota} u\|_{L^2}^2\\\nonumber
	\leq& C\left(\|\nabla a\|_{L^\infty}^2\|\Li u\|_{L^2}^2+\|\nabla u\|_{L^\infty}^2\|\Li a\|_{L^2}^2\right)+\int_{\mathbb{T}^n} \Li  f_1 \cdot \Li a \mathrm{~d} x\\\label{3.31}
	&+\int_{\mathbb{T}^n} \Li  f_2 \cdot \Li u \mathrm{~d} x+\int_{\mathbb{T}^n} \Li  f_3 \cdot \Li B \mathrm{~d} x.
\end{align}
To estimate the term involving $f_1$, by integration by parts and Lemma \ref{commutator}, we first obtain
\begin{align}\nonumber
	&\left|\int_{\mt^n}\Li \left(u \cdot \nabla a\right)  \cdot \Li a \mathrm{~d} x\right|\\\nonumber
	\leq& 	\left|\int_{\mt^n}\left(\Li \left(u \cdot \nabla a\right) -u\cdot \nabla \Li a \right) \cdot \Li a \mathrm{~d} x\right|+	\left|\int_{\mt^n}u\cdot \nabla \Li a \cdot \Li a \mathrm{~d} x\right|\\\nonumber
	\leq& C \|\Li a\|_{L^2}\left(\|\nabla u\|_{L^\infty}\|\Li a\|_{L^2}+\|\nabla a\|_{L^\infty}\|\Li u\|_{L^2}\right)+C\|\nabla u\|_{L^\infty}\|\Li a\|_{L^2}^2\\\label{3.4}
	\leq& C\|\nabla u\|_{L^\infty}\|\Li a\|_{L^2}^2+C \|\Li a\|_{L^2}\|\nabla a\|_{L^\infty}\|\Li u\|_{L^2}
\end{align}
and
\begin{align}\label{3.5}
	\left|\int_{\mt^n} \Li \left(a \nabla \cdot u\right) \cdot \Li a \mathrm{~d} x\right|\leq C\|\Li a\|_{L^2}\left(\|\Li a\|_{L^2}\|\nabla u\|_{L^\infty}+\|a\|_{L^\infty}\|\nabla\Lambda^{\iota}u\|_{L^2}\right),
\end{align}
then by recalling \eqref{nonlinearity} and combining \eqref{3.4} and \eqref{3.5}, it yields that
\begin{align}\nonumber
	&\left|\int_{\mt^n} \Li  f_1 \cdot \Li a \mathrm{~d} x\right|\\\nonumber
	=&	\left|\int_{\mt^n} \Li \left(u \cdot \nabla a+a \nabla \cdot u\right) \cdot \Li a \mathrm{~d} x\right|\\\label{3.51}
	\leq& 	\left|\int_{\mt^n}\Li \left(u \cdot \nabla a\right)  \cdot \Li a \mathrm{~d} x\right|+	\left|\int_{\mt^n} \Li \left(a \nabla \cdot u\right) \cdot \Li a \mathrm{~d} x\right|\\\nonumber
	\leq& C\|\nabla u\|_{L^\infty}\|\Li a\|_{L^2}^2+C\|\Li a\|_{L^2}\left(\|\nabla a\|_{L^\infty}\|\Lambda^{\iota}u\|_{L^2}+\|a\|_{L^\infty}\|\nabla\Lambda^{\iota}u\|_{L^2}\right)\\\nonumber
	\leq& \dfrac{c_0 \mu}{128}\|\nabla\Lambda^{\iota}u\|_{L^2}^2+C\left(\|\nabla u\|_{L^\infty}\|\Li a\|_{L^2}^2+\|\Li a\|_{L^2}\|\nabla a\|_{L^\infty}\|\Lambda^{\iota}u\|_{L^2}+\|a\|_{L^\infty}^2\|\Li a\|_{L^2}^2\right).
\end{align}

Similarly, one can estimate the term involving $f_3$ as
\begin{align}\label{3.52}
	&\left|\int_{\mathbb{T}^n} \Li  f_3 \cdot \Li B \mathrm{~d} x\right|\\
	\leq& \dfrac{c_0 \mu}{128}\|\nabla \Lambda^{\iota}u\|_{L^2}^2+C\left(\|\nabla u\|_{L^\infty}\|\Li B\|_{L^2}^2+\|\Li B\|_{L^2}\|\nabla B\|_{L^\infty}\|\Lambda^{\iota}u\|_{L^2}+\|B\|_{L^\infty}^2\|\Li B\|_{L^2}^2\right).\nonumber
\end{align}

It remains to deal with the term involving $f_2$. To achieve this, we first recall the definition of $f_2$ in \eqref{nonlinearity} as:
\begin{align*}
	\begin{aligned}
		f_2 \stackrel{\text { def }}{=} & - u \cdot \nabla  u+ B \cdot \nabla  B+
		\nabla( B \cdot  B)+k(a) \nabla a+\mu(\nabla I(a)) \nabla  u \\
		& +\lambda(\nabla I(a)) \nabla \cdot u-I(a)\left(\tb \cdot \nabla  B+ B \cdot \nabla  B-\nabla(\tb \cdot B) -\nabla( B \cdot  B)\right),
	\end{aligned}
\end{align*}
and then there holds
\begin{align}\label{3.6}
	\left|\int_{\mathbb{T}^n} \Li  f_2 \cdot \Li u \mathrm{~d} x\right|\leq K_1+K_2+K_3+K_4+K_5+K_6+K_7,
\end{align}
where
\begin{align*}
	&	K_1=\left|\int_{\mathbb{T}^n} \Li  (u\cdot\nabla u) \cdot \Li u \mathrm{~d} x\right|,\\
	&K_2=\left|\int_{\mathbb{T}^n} \Li   \left(B \cdot \nabla  B+
	\nabla( B \cdot  B)\right) \cdot \Li u \mathrm{~d} x\right|,\\
	&K_3= \left|\int_{\mathbb{T}^n} \Li  (k(a) \nabla a) \cdot \Li u \mathrm{~d} x\right|,\\
	&K_4=\left|\int_{\mathbb{T}^n} \Li  (\mu(\nabla I(a)) \nabla  u) \cdot \Li u \mathrm{~d} x\right|,\\
	&	K_5= \left|\int_{\mathbb{T}^n} \Li  (\lambda(\nabla I(a)) \nabla \cdot u) \cdot \Li u \mathrm{~d} x\right|,\\
	&K_6=\left|\int_{\mathbb{T}^n} \Li \left( I(a)\left(\tb \cdot \nabla  B-\nabla(\tb \cdot B) \right)\right) \cdot \Li u \mathrm{~d} x\right|,\\
	&	K_7= \left|\int_{\mathbb{T}^n} \Li  \left(I(a)\left( B \cdot \nabla  B-\nabla( B \cdot  B)\right) \right)\cdot \Li u \mathrm{~d} x\right|.
\end{align*}

Similar to \eqref{3.1}, the term $K_1$ can be bounded as
\begin{align*}
	K_1\leq C\|\nabla u\|_{L^\infty}\|\Li u\|_{L^2}^2.
\end{align*}
Considering that $B$ is divergence-free, we have
\begin{align*}
	K_2&=\left|\int_{\mathbb{T}^n} \Li  \san \left(B\otimes B\right) \cdot \Li u \mathrm{~d} x\right|+\left|\int_{\mathbb{T}^n} \Li \nabla \left(B\cdot B\right) \cdot \Li u \mathrm{~d} x\right|\\
	&\leq C \|\nabla\Lambda^{\iota}u\|_{L^2}\|\Li B\|_{L^2}\|B\|_{L^\infty}\\
	&\leq \dfrac{c_0\mu}{128}\|\nabla\Lambda^{\iota}u\|_{L^2}^2+C\|\Li B\|_{L^2}^2\|B\|_{L^\infty}^2.
\end{align*}
By integration by parts, Lemma \ref{commutator} and \eqref{Ia}, it yields that
\begin{align*}
	K_3&= \left|\int_{\mathbb{T}^n} \Li  (k(a) \nabla a) \cdot \Li u \mathrm{~d} x\right|\\
&= \left|\int_{\mathbb{T}^n}\san\nabla\Delta^{-1}\left(\Li  (k(a) \nabla a)\right) \cdot \Li u \mathrm{~d} x\right|\\
	&\leq C\|\nabla\Lambda^{\iota}u\|_{L^2}\left(\|\Li a\|_{L^2}\|k(a)\|_{L^\infty}+\|\Lambda^{\iota-1}k(a)\|_{L^2}\|\nabla a\|_{L^\infty}\right)\\
	&\leq \dfrac{c_0\mu}{128}\|\nabla\Lambda^{\iota}u\|_{L^2}^2+C\left(\|\Li a\|_{L^2}^2\|a\|_{L^\infty}^2+\|\Lambda^{\iota-1}a\|_{L^2}^2\|\nabla a\|_{L^\infty}^2\right).
\end{align*}
Using similar methods to estimate $K_3$, we can get
\begin{align*}
	K_4&\leq C\|\nabla\Lambda^{\iota}u\|_{L^2}\left(\|\Li I(a)\|_{L^2}\|\nabla u\|_{L^\infty}+\|\Lambda^{\iota}u\|_{L^2}\|\nabla I(a)\|_{L^\infty}\right)\\
	&\leq \dfrac{c_0\mu}{128}\|\nabla\Lambda^{\iota}u\|_{L^2}^2+C\left(\|\Li a\|_{L^2}^2\|\nabla u\|_{L^\infty}^2+\|\Lambda^{\iota}u\|_{L^2}^2\|\nabla a\|_{L^\infty}^2\right),
\end{align*}
\begin{align*}
	K_5\leq \dfrac{c_0\mu}{128}\|\nabla\Lambda^{\iota}u\|_{L^2}^2+C\left(\|\Li a\|_{L^2}^2\|\nabla u\|_{L^\infty}^2+\|\Lambda^{\iota}u\|_{L^2}^2\|\nabla a\|_{L^\infty}^2\right),
\end{align*}
and
\begin{align*}
	K_6&\leq C\|\nabla\Lambda^{\iota}u\|_{L^2}\left(\|\Lambda^{\iota-1} I(a)\|_{L^2}\|\nabla B\|_{L^\infty}+\|\Lambda^{\iota}B\|_{L^2}\|I(a)\|_{L^\infty}\right)\\
	&\leq \dfrac{c_0\mu}{128}\|\nabla\Lambda^{\iota}u\|_{L^2}^2+C\left(\|\Lambda^{\iota-1} a\|_{L^2}^2\|\nabla B\|_{L^\infty}^2+\|\Lambda^{\iota}B\|_{L^2}^2\|a\|_{L^\infty}^2\right).
\end{align*}
Employing similar tools in estimating $K_2$, there holds
\begin{align*}
	\|\Lambda^{\iota-1}\left(B \cdot \nabla  B-\nabla( B \cdot  B)\right)\|_{L^2}\leq C\|\Li B\|_{L^2}\|B\|_{L^\infty},
\end{align*}
and thus
\begin{align*}
	K_7&\leq C\|\nabla\Lambda^{\iota}u\|_{L^2}\left(\|I(a)\|_{L^\infty}\|\Li B\|_{L^2}\|B\|_{L^\infty}+\|\Lambda^{\iota-1}I(a)\|_{L^2}\|B\|_{L^\infty}\|\nabla B\|_{L^\infty}\right)\\
	&\leq \dfrac{c_0\mu}{128}\|\nabla\Lambda^{\iota}u\|_{L^2}^2+C\left(\|\Lambda^{\iota-1}a\|_{L^2}^2\|B\|_{L^\infty}^2\|\nabla B\|_{L^\infty}^2+\|a\|_{L^\infty}^2\|\Li B\|_{L^2}^2\|\nabla B\|_{L^\infty}^2\right).
\end{align*}
Collecting all estimates above leads to
\begin{align}\label{3.7}
	&\left|\int_{\mathbb{T}^n} \Li  f_2 \cdot \Li u \mathrm{~d} x\right|\\
	\leq& \dfrac{6c_0\mu}{128}\|\nabla\Lambda^{\iota}u\|_{L^2}^2
	+C\|\Li (a,u,B)\|_{L^2}^2\left(\|\nabla u\|_{L^\infty}+\|(a,u,B)\|_{W^{1,\infty}}^2+\|(a,B)\|_{W^{1,\infty}}^4\right).\nonumber
\end{align}

Finally, by combining \eqref{3.51}, \eqref{3.52} and \eqref{3.7}, we have
\begin{align}\label{3.53}
	&\left|	\int_{\mathbb{T}^n} \Li  f_1 \cdot \Li a \mathrm{~d} x\right|+\left|\int_{\mathbb{T}^n} \Li  f_2 \cdot \Li u \mathrm{~d} x\right|+\left|\int_{\mathbb{T}^n} \Li  f_3 \cdot \Li B \mathrm{~d} x\right|\\
	\leq& \dfrac{c_0\mu}{16}\|\nabla\Li u\|_{L^2}^2+C\|\Li (a,u,B)\|_{L^2}^2\left(\| (a,u,B)\|_{W^{1,\infty}}+\|(a,u,B)\|_{W^{1,\infty}}^2+\|(a,B)\|_{W^{1,\infty}}^4\right).\nonumber
\end{align}
Plugging \eqref{3.53} into \eqref{3.31} and applying the Poincaré-type inequalities \eqref{bpoincare} and \eqref{apoincare} for $B$ and $a$, it follows that
\begin{align}\nonumber
			&\frac{1}{2}	\frac{\dd}{\dd t}\|\Li  (a,u,B)\|_{L^2}^2+\frac{1}{4}c_0\lambda \|\Lambda^{\iota}\san u\|_{L^2}^2+\frac{3}{16}c_0\mu \|\nabla\Lambda^{\iota} u\|_{L^2}^2 \\\nonumber
\leq& C\|\Li (a,u,B)\|_{L^2}^2\left(\| (a,u,B)\|_{W^{1,\infty}}+\|(a,u,B)\|_{W^{1,\infty}}^2+\|(a,B)\|_{W^{1,\infty}}^4\right).
\end{align}

{\bf When $\iota=0$}, the difference from the case $\iota\geq1$ only lies in the estimations of terms $K_6$ and $K_7$. In this case, we estimate them as
\begin{align*}
	K_6=\left|\int_{\mathbb{T}^n}\left( I(a)\left(\tb \cdot \nabla  B-\nabla(\tb \cdot B) \right)\right) \cdot u \mathrm{~d} x\right|\leq C\|u\|_{L^2}\|a\|_{L^2}\|\nabla B\|_{L^\infty}
\end{align*}
and
\begin{align*}
	K_7&= \left|\int_{\mathbb{T}^n}  \left(I(a)\left( B \cdot \nabla  B-\nabla( B \cdot  B)\right) \right)\cdot u \mathrm{~d} x\right|\\
	&\leq C\|u\|_{L^2}\|a\|_{L^2}\|\nabla B\|_{L^\infty}\|B\|_{L^\infty}.
\end{align*}
Regarding the remaining terms, it suffices to follow a similar approach as in the case $\iota\geq 1$. Hence, we can deduce that
\begin{align}\nonumber
	&\frac{1}{2}	\frac{\dd}{\dd t}\|  (a,u,B)\|_{L^2}^2+\frac{1}{4}c_0\lambda \|\san u\|_{L^2}^2+\frac{3}{16}c_0\mu \|\nabla u\|_{L^2}^2 \\\label{3.8}
	\leq& C\|(a,u,B)\|_{L^2}^2\left(\| (a,u,B)\|_{W^{1,\infty}}+\|(a,u,B)\|_{W^{1,\infty}}^2+\|(a,B)\|_{W^{1,\infty}}^4\right).
\end{align}
Combining \eqref{3.7} and \eqref{3.8} leads to \eqref{3.00}, which completes all the proof.
\end{proof}

\subsection{The dissipation mechanism of $B$}

\,\,\,\,\,\,\,\,As mentioned in \cite{zhai-3m-compressible-nonresistive}, the original system \eqref{original} only have dissipation in the equations of velocity field, while the equations of magnetic field equation has no dissipation or damping. However, thanks to the Diophantine condition \eqref{Diophantine} satisfied by the background magnetic field, there is hidden dissipative mechanism in the perturbation system \eqref{equation0}-\eqref{h2}, which will be summarized in Lemma \ref{lem3}. In fact, for compressible MHD equations, because the equations of $u$ in \eqref{equation}$_1$ become much more complicated and it would be
exceptionally challenging to estimate them. To this end, we will construct the following coupling terms for subsequent estimates
\begin{align*}
	-\frac{\dd}{\dd t}\int_{\mathbb{T}^n}\left(\tb\cdot\nabla B\right)\cdot \Lambda^s \mathbb{P}u\dd x-\frac{\dd}{\dd t}\int_{\mathbb{T}^n}\left(\tb\cdot\nabla B\right)\cdot \mathbb{P}u\dd x,
\end{align*}
which is quite different from the incompressible case in \cite{preprite}. In this case, applying the Leray projection operator $\mathbb{P}$ on both sides of \eqref{equation}$_1$, it yields that
	\begin{align}
		\partial_t \mathbb{P}  u-\mu \Delta \mathbb{P}  u=  \tb \cdot \nabla  B+\mathbb{P} f_4,
\end{align}
where
\begin{align}
		f_4 \stackrel{\text { def }}{=} & - u \cdot \nabla  u+ B \cdot \nabla  B+ 	\nabla( B \cdot  B)+k(a) \nabla a -I(a)\left(\mu \Delta  u+\lambda \nabla \operatorname{div}  u\right)\nonumber\\
		&-I(a)\left(\tb \cdot \nabla  B+ B \cdot \nabla  B-\nabla(\tb \cdot B) -\nabla( B \cdot  B)\right).\label{f41}
\end{align}

Now, we get to state our main lemma in this subsection.

	\begin{lemma}\label{lem3}
	For any $s\geq 0$, the following estimate holds
	\begin{align}\nonumber
		&-\frac{\dd}{\dd t}\int_{\mathbb{T}^n}\left(\tb\cdot\nabla B\right)\cdot \Lambda^s \mathbb{P}u\dd x	-\frac{\dd}{\dd t}\int_{\mathbb{T}^n}\left(\tb\cdot\nabla B\right)\cdot \mathbb{P}u\dd x+c_1\|\tb\cdot\nabla B\|_{H^{\frac{s}{2}}}^2\\\nonumber
		\leq& \left(\dfrac{1}{2}\mu^2+2|\tb|^2\right)\|\nabla u\|_{H^{\frac{s}{2}+1}}+C\|(a,u,B)\|_{H^{\frac{s}{2}+1}}^2\left(\| (a,u,B)\|_{W^{1,\infty}}+\| (a,B)\|_{W^{1,\infty}}^2\right)\\\label{3.90}
		&+C\|B\|_{H^{\frac{s}{2}+1}}\left(\|a\|_{L^\infty}\|\nabla u\|_{H^{\frac{s}{2}+1}}+\|a\|_{H^\frac{s}{2}}\|\nabla^2 u\|_{L^\infty}+\|\nabla  u\|_{H^{\frac{s}{2}+1}}\|u\|_{L^\infty}\right).
	\end{align}
\end{lemma}
\begin{proof}
	Trough some basic calculations, it follows that
	\begin{align}\nonumber
		-\frac{\dd}{\dd t}\int_{\mathbb{T}^n}\left(\tb\cdot\nabla B\right)\cdot \Lambda^s \mathbb{P}u\dd x&=-\int_{\mathbb{T}^n}\left(\tb\cdot\nabla \partial_t B\right)\cdot \Lambda^s \mathbb{P}u-\left(\tb\cdot\nabla B\right)\cdot \Lambda^s \partial_t \mathbb{P}u\dd x\\
		&\triangleq J_1+J_2.\label{jj}
	\end{align}
	
	To estimate $J_1$, we first use \eqref{equation}$_3$ to update it as
	\begin{align}\nonumber
		J_1&=-\int_{\mathbb{T}^n}\left(\tb\cdot\nabla \left(\tb \cdot \nabla  u-\tb \nabla \cdot u+f_3\right)\right)\cdot \Lambda^s \mathbb{P}u\dd x\\\label{J1}
		&\leq 2|\tb|^2\|\nabla\Lambda^{\frac{s}{2}}u\|_{L^2}^2+\left|\int_{\mathbb{T}^n}\left(\tb\cdot\nabla f_3\right)\cdot \Lambda^s \mathbb{P}u\dd x\right|.
	\end{align}
Noticing $\nabla \cdot B=0$, we can rewrite $f_3$ as:
	\begin{align*}
		f_3&=- u \cdot \nabla  B+ B \cdot \nabla  u- B \nabla \cdot u\\
&= - u \cdot \nabla  B+ B \cdot \nabla  u- B \nabla \cdot u+u\nabla\cdot B\\
		&=\nabla\times (u\times B),
	\end{align*}
which imply
	\begin{align}
		&\left|\int_{\mathbb{T}^n}\left(\tb\cdot\nabla f_3\right)\cdot \Lambda^s \mathbb{P}u\dd x\right|\nonumber\\
=&\left|\int_{\mathbb{T}^n}\left(\tb\cdot\nabla \left[\nabla\times \left(u\times B\right)\right]\right)\cdot \Lambda^s \mathbb{P}u\dd x\right|\nonumber\\
		\leq& |\tb| \|\nabla \Lambda^\frac{s}{2}u\|_{L^2}\left(\|\nabla \Lambda^\frac{s}{2}u\|_{L^2}\|B\|_{L^\infty}+\|\nabla \Lambda^\frac{s}{2}B\|_{L^2}\|u\|_{L^\infty}\right).\label{J01}
	\end{align}
Plugging \eqref{J01} into \eqref{J1}, we have
	\begin{align}\label{j1}
		J_1\leq 2|\tb|^2\|\nabla\Lambda^{\frac{s}{2}+1}u\|_{L^2}^2+C\left(\|\Lambda^{\frac{s}{2}+1}u\|_{L^2}^2+\|\Lambda^{\frac{s}{2}+1}B\|_{L^2}^2\right)\left(\| u\|_{L^\infty}+\|B\|_{L^\infty}\right).
	\end{align}
	
	Similar to estimating $J_1$, one can use \eqref{equation}$_2$ and integrations by parts to get
	\begin{align}\nonumber
		J_2&=-\int_{\mathbb{T}^n}\left(\tb\cdot\nabla B\right)\cdot \Lambda^s\left(\mu \Delta \mathbb{P}  u+  \tb \cdot \nabla  B+\mathbb{P} f_4\right)\dd x\notag\\\nonumber
		&=-\left\|\Lambda^\frac{s}{2}\left(\tb\cdot \nabla B\right)\right\|_{L^2}^2-\int_{\mathbb{T}^n}\left(\tb\cdot\nabla B\right)\cdot \Lambda^s\left(\mu \Delta \mathbb{P}  u+\mathbb{P} f_4\right)\dd x\\\label{J2}
		&\leq -\frac{1}{2}\left\|\Lambda^\frac{s}{2}\left(\tb\cdot \nabla B\right)\right\|_{L^2}^2+\dfrac{1}{2}\mu^2\|\nabla \Lambda^{\frac{s}{2}+1}u\|_{L^2}^2+\left|\int_{\mathbb{T}^n}\left(\tb\cdot \nabla B\right)\cdot \Lambda^s \mathbb{P}f_4 \dd x\right|.
	\end{align}
Recalling the definition of $f_4$ in \eqref{f41}, we can rewrite  the last term of \eqref{J2} as
	\begin{align}\label{J}
		\left|\int_{\mathbb{T}^n}\left(\tb\cdot \nabla B\right)\cdot \Lambda^s \mathbb{P}f_4 \dd x\right|\leq J_{21}+J_{22}+J_{23}+J_{24}+J_{25}+J_{26}+J_{27},
	\end{align}
where
	\begin{align*}
		&J_{21}=	\left|\int_{\mathbb{T}^n}\left(\tb\cdot \nabla B\right)\cdot \Lambda^s \mathbb{P}\left(u \cdot \nabla  u\right)\dd x\right|,\\
		&J_{22}=\left|\int_{\mathbb{T}^n}\left(\tb\cdot \nabla B\right)\cdot \Lambda^s \mathbb{P}\left(B \cdot \nabla  B\right)\dd x\right|,\\
		&J_{23}=\left|\int_{\mathbb{T}^n}\left(\tb\cdot \nabla B\right)\cdot \Lambda^s \mathbb{P}\left(\nabla( B \cdot  B)\right)\dd x\right|,\\
		&J_{24}=\left|\int_{\mathbb{T}^n}\left(\tb\cdot \nabla B\right)\cdot \Lambda^s \mathbb{P}\left(k(a) \nabla a \right)\dd x\right|,\\
		&J_{25}=\left|\int_{\mathbb{T}^n}\left(\tb\cdot \nabla B\right)\cdot \Lambda^s \mathbb{P}\left(I(a)\left(\mu \Delta  u+\lambda \nabla \operatorname{div}  u\right)\right)\dd x\right|,\\
		&J_{26}=\left|\int_{\mathbb{T}^n}\left(\tb\cdot \nabla B\right)\cdot \Lambda^s \mathbb{P}\left(I(a)\left( B \cdot \nabla  B-\nabla( B \cdot  B)\right)\right)\dd x\right|,\\
		&J_{27}=\left|\int_{\mathbb{T}^n}\left(\tb\cdot \nabla B\right)\cdot \Lambda^s \mathbb{P}\left(I(a)\left(\tb \cdot \nabla  B-\nabla(\tb \cdot B)\right)\right)\dd x\right|.
	\end{align*}
Making use of integration by parts and Lemma \ref{commutator}, it follows that
	\begin{align*}
		J_{21}&=	\left|\int_{\mathbb{T}^n}\left(\tb\cdot \nabla B\right)\cdot \Lambda^s \mathbb{P}\left(u \cdot \nabla  u\right)\dd x\right|\\
		&\leq C
		\left\|\Lambda^{\frac{s}{2}-1}\left(\tb\cdot \nabla B\right)\right\|_{L^2}\|\Lambda^{\frac{s}{2}+1}(u\cdot\nabla u)\|_{L^2}\\
		&\leq C\|\Lambda^\frac{s}{2}B\|_{L^2}\left(\|\Lambda^{\frac{s}{2}+1} u\|_{L^2}\|\nabla u\|_{L^\infty}+\|\nabla\Lambda^{\frac{s}{2}+1} u\|_{L^2}\|u\|_{L^\infty}\right)
	\end{align*}
and
	\begin{align*}
		J_{22}+J_{23}&=\left|\int_{\mathbb{T}^n}\left(\tb\cdot \nabla B\right)\cdot \Lambda^s \mathbb{P}\left(B \cdot \nabla  B\right)\dd x\right|+\left|\int_{\mathbb{T}^n}\left(\tb\cdot \nabla B\right)\cdot \Lambda^s \mathbb{P}\left(\nabla( B \cdot  B)\right)\dd x\right|\\
		&\leq C
		\left\|\Lambda^{\frac{s}{2}}\left(\tb\cdot \nabla B\right)\right\|_{L^2}\left(\|\Lambda^{\frac{s}{2}}\nabla\left(B\cdot B\right)\|_{L^2}+\|\Lambda^{\frac{s}{2}}\nabla\left(B\otimes B\right)\|_{L^2}\right)\\
		&\leq C\|\Lambda^{\frac{s}{2}+1}B\|_{L^2}\|\Lambda^{\frac{s}{2}+1}B\|_{L^2}\|B\|_{L^\infty}.
	\end{align*}
According to the definition of pressure, $P=P(\rho)=P(a+1)$ is a smooth function satisfying $P'(1)=1$. Setting
	\begin{align}\label{ga}
		g(a)=\int_0^a \left(\dfrac{P'(r+1)}{r+1}-1\right)\dd r,
	\end{align}
	then it clearly holds that $g(0)=g'(0)=0$ and $g(a)\sim a^2$. Thanks to \eqref{ga}, we can deduce that $k(a)\nabla a=\nabla g(a)$ and therefore
	\begin{align*}
		\mathbb{P}(k(a)\nabla a)=0 \Rightarrow J_{24}=\left|\int_{\mathbb{T}^n}\left(\tb\cdot \nabla B\right)\cdot \Lambda^s \mathbb{P}\left(k(a) \nabla a \right)\dd x\right|=0.
	\end{align*}
Employing \eqref{Ia}, \eqref{Ia1} and Lemma \ref{commutator}, one has
	\begin{align*}
		J_{25}&=\left|\int_{\mathbb{T}^n}\left(\tb\cdot \nabla B\right)\cdot \Lambda^s \mathbb{P}\left(I(a)\left(\mu \Delta  u+\lambda \nabla \operatorname{div}  u\right)\right)\dd x\right|\\
		&\leq C\|\Lambda^{\frac{s}{2}+1}B\|_{L^2}\left(\|I(a)\|_{L^\infty}\|\nabla^2 \Lambda^{\frac{s}{2}}u\|_{L^2}+\|\Lambda^{\frac{s}{2}}I(a)\|_{L^2}\|\nabla^2 u\|_{L^\infty}\right)\\
		&\leq C\|\Lambda^{\frac{s}{2}+1}B\|_{L^2}\left(\|a\|_{L^\infty}\|\nabla^2 \Lambda^{\frac{s}{2}}u\|_{L^2}+\|\Lambda^{\frac{s}{2}}a\|_{L^2}\|\nabla^2 u\|_{L^\infty}\right),
	\end{align*}
	\begin{align*}
		J_{26}&=\left|\int_{\mathbb{T}^n}\left(\tb\cdot \nabla B\right)\cdot \Lambda^s \mathbb{P}\left(I(a)\left[ B \cdot \nabla  B-\nabla( B \cdot  B)\right]\right)\dd x\right|\\
		&\leq C\|\Lambda^{\frac{s}{2}+1}B\|_{L^2}\left(\|I(a)\|_{L^\infty}\|\Lambda^{\frac{s}{2}+1}(B\otimes B)\|_{L^2}+\|\Lambda^{\frac{s}{2}}I(a)\|_{L^2}\|\nabla (B\otimes B)\|_{L^\infty}\right)\\
		&\leq C\|\Lambda^{\frac{s}{2}+1}B\|_{L^2}\left(\|a\|_{L^\infty}\|\Lambda^{\frac{s}{2}+1}B\|_{L^2}\|B\|_{L^\infty}+\|\Lambda^{\frac{s}{2}}a\|_{L^2}\|\nabla B\|_{L^\infty}\|B\|_{L^\infty}\right)
	\end{align*}
and
	\begin{align*}
		J_{27}&=\left|\int_{\mathbb{T}^n}\left(\tb\cdot \nabla B\right)\cdot \Lambda^s \mathbb{P}\left(I(a)\left[\tb \cdot \nabla  B-\nabla(\tb \cdot B)\right]\right)\dd x\right|	\\
		&\leq C\|\Lambda^{\frac{s}{2}+1}B\|_{L^2}\left(\|a\|_{L^\infty}\|\Lambda^{\frac{s}{2}+1}B\|_{L^2}+\|\Lambda^{\frac{s}{2}}a\|_{L^2}\|\nabla B\|_{L^\infty}\right).
	\end{align*}

Collecting the estimates above and plugging the resultants into \eqref{J}, we obtain
	\begin{align}\label{3.9}
		&	\left|\int_{\mathbb{T}^n}\left(\tb\cdot \nabla B\right)\cdot \Lambda^s \mathbb{P}f_4 \dd x\right|\\
		\leq& C\|\Lambda^{\frac{s}{2}+1}B\|_{L^2}\left(\|a\|_{L^\infty}\|\nabla^2 \Lambda^{\frac{s}{2}}u\|_{L^2}+\|\Lambda^{\frac{s}{2}}a\|_{L^2}\|\nabla^2 u\|_{L^\infty}+\|\nabla\Lambda^{\frac{s}{2}+1} u\|_{L^2}\|u\|_{L^\infty}\right)\nonumber\\
		&+ C\|\Lambda^{\frac{s}{2}+1}B\|_{L^2}\|\Lambda^{\frac{s}{2}+1}(a,u,B)\|_{L^2}\left(\| (a,u,B)\|_{W^{1,\infty}}+\| B\|_{W^{1,\infty}}^2+\| a\|_{W^{1,\infty}}^2\right),\nonumber
	\end{align}
which further implies, after combining with \eqref{J2}, that
	\begin{align}\nonumber
		J_2&\leq -\frac{1}{2}\left\|\Lambda^\frac{s}{2}\left(\tb\cdot \nabla B\right)\right\|_{L^2}^2+\dfrac{1}{2}\mu^2\|\nabla\Lambda^{\frac{s}{2}+1}u\|_{L^2}^2\\\nonumber
		&~~+ C\|\Lambda^{\frac{s}{2}+1}B\|_{L^2}\left(\|a\|_{L^\infty}\|\nabla^2 \Lambda^{\frac{s}{2}}u\|_{L^2}+\|\Lambda^{\frac{s}{2}}a\|_{L^2}\|\nabla^2 u\|_{L^\infty}+\|\nabla\Lambda^{\frac{s}{2}+1} u\|_{L^2}\|u\|_{L^\infty}\right)\\\label{j2}
		&~~+C\|\Lambda^{\frac{s}{2}+1}(a,u,B)\|_{L^2}^2\left(\| (a,u,B)\|_{W^{1,\infty}}+\| B\|_{W^{1,\infty}}^2+\| a\|_{W^{1,\infty}}^2\right).
	\end{align}

Inserting \eqref{j1} and \eqref{j2} into \eqref{jj}, it follows that
		\begin{align}\label{3.10}
		&-\frac{\dd}{\dd t}\int_{\mathbb{T}^n}\left(\tb\cdot\nabla B\right)\cdot \Lambda^s \mathbb{P}u\dd x\nonumber\\
\leq& -\frac{1}{2}\left\|\Lambda^\frac{s}{2}\left(\tb\cdot \nabla B\right)\right\|_{L^2}^2+\left(\dfrac{1}{2}\mu^2+2|\tb|^2\right)\|\nabla\Lambda^{\frac{s}{2}+1}u\|_{L^2}^2\nonumber\\
		&+ C\|\Lambda^{\frac{s}{2}+1}B\|_{L^2}\left(\|a\|_{L^\infty}\|\nabla^2 \Lambda^{\frac{s}{2}}u\|_{L^2}+\|\Lambda^{\frac{s}{2}}a\|_{L^2}\|\nabla^2 u\|_{L^\infty}+\|\nabla\Lambda^{\frac{s}{2}+1} u\|_{L^2}\|u\|_{L^\infty}\right)\nonumber\\
		&+C\|\Lambda^{\frac{s}{2}+1}(a,u,B)\|_{L^2}^2\left(\| (a,u,B)\|_{W^{1,\infty}}+\| B\|_{W^{1,\infty}}^2+\| a\|_{W^{1,\infty}}^2\right).
	\end{align}
For the special case $s=0$ of \eqref{3.10}, one can use similar approach to get
\begin{align}\nonumber
	&-\frac{\dd}{\dd t}\int_{\mathbb{T}^n}\left(\tb\cdot\nabla B\right)\cdot  \mathbb{P}u\dd x\\\nonumber
\leq& -\frac{1}{2}\left\|\left(\tb\cdot \nabla B\right)\right\|_{L^2}^2+\left(\dfrac{1}{2}\mu^2+2|\tb|^2\right)\|\nabla\Lambda u\|_{L^2}^2\\\nonumber
	&+ C\|\Lambda B\|_{L^2}\left(\|a\|_{L^\infty}\|\nabla^2 u\|_{L^2}+\|a\|_{L^2}\|\nabla^2 u\|_{L^\infty}+\|\nabla \Lambda u\|_{L^2}\|u\|_{L^\infty}\right)\\\label{3.11}
	&+C\|\Lambda(a,u,B)\|_{L^2}^2\left(\| (a,u,B)\|_{W^{1,\infty}}+\| B\|_{W^{1,\infty}}^2+\| a\|_{W^{1,\infty}}^2\right).
\end{align}
Finally, applying the following inequality
\begin{align*}
	-\frac{1}{2}\left\|\Lambda^\frac{s}{2}\left(\tb\cdot \nabla B\right)\right\|_{L^2}^2-\frac{1}{2}\|\tb\cdot\nabla B\|_{L^2}^2\leq -c_1\|\tb\cdot\nabla B\|_{H^{\frac{s}{2}}}^2,
\end{align*}
we can infer \eqref{3.90} from \eqref{3.10} and \eqref{3.11}, that finish all the proof.
\end{proof}

\subsection{The dissipation mechanism of $a+\tb\cdot B $}
\,\,\,\,\,\,\,\,In the previous subsection, we focused on the dissipation mechanism of $B$ arising from the background magnetic field. In this subsection, we will shift our focus on the dissipation mechanism of quantity $a+\tb\cdot B$. As a matter of fact, we will inherit the idea by Wu and Zhai in \cite{zhai-3m-compressible-nonresistive} and present it here with minor modification and our analysis for integrity of this paper and readers' convenience.

First, taking the inner product on both sides of \eqref{equation}$_3$ with $\tb$, it follows that
\begin{align*}
	\partial_t\left(\tb\cdot B\right)=\left(\tb\cdot\nabla \right)u\cdot\tb-|\tb|^2\nabla\cdot u+f_3\cdot\tb,
\end{align*}
which together with \eqref{equation}$_1$ yield
\begin{align}\label{d}
	\partial_t\left(a+\tb\cdot B \right)=\left(\tb\cdot\nabla \right)u\cdot\tb-\left(|\tb|^2+1\right)\nabla\cdot u+f_1+f_3\cdot\tb.
\end{align}
In \eqref{d}, one can notice that $\nabla\cdot u$ is a hidden dissipative term, and satisfies $\san{\nabla \san}{\Delta}^{-1}u=\nabla\cdot\mathbb{D}u=\nabla\cdot u$. Now, by recalling the following equations satisfied by $\mathbb{D}u$ in \eqref{nabla}
	\begin{align}\label{Du}
			\partial_t \mathbb{D}  u-(\lambda+\mu) \Delta \mathbb{D}  u+\nabla a+\nabla(  \tb \cdot  B)=\mathbb{D} f_4.
\end{align}
we can discover that \eqref{Du} contains the term $\nabla\left(a+\tb\cdot B\right)$, which involves a wave structure for $a+\tb\cdot B$ and $\mathbb{D}u$. To utilize this point better, we set
\begin{align}\label{definitionsigma}
	w \stackrel{\text { def }}{=} a+\tb\cdot B, \quad \sigma \stackrel{\text { def }}{=} \mathbb{D} u-\dfrac{1}{\lambda+\mu}\Delta^{-1}\nabla w,
\end{align}
which satisfies
\begin{align}\label{poincarew}
	\int_{\mt^n}a(x,t)\dd x=\int_{\mt^n}B(x,t)\dd x=0 \Rightarrow \int_{\mt^n}w(x,t) \dd x=0,
\end{align}
for any $t>0$, due to \eqref{rho0} and \eqref{b0}. According to the theory of Fourier analysis, the inverse Laplacian $\Delta^{-1}$ of a smooth periodic function $f$ with mean zero is given by the Fourier series formula
\begin{equation}\label{inversedelta}
\Delta^{-1}f(x)=\sum_{|k|\neq 0,k\in \mathbb{Z}^n}\frac{1}{-4\pi|k|^2}\hat{f}(k)e^{2\pi ik\cdot x}.
\end{equation}
Considering ${\nabla \san}\,u(x,t)$ and $\nabla w(x,t)$ have mean zero automatically, we can clearly infer that
\begin{align*}
	\int_{\mt^n}\mathbb{D} u\dd x=\int_{\mt^n}{\nabla\Delta^{-1} \san}\,u(x,t)\dd x=0,\quad \int_{\mt^n} \Delta^{-1}{\nabla}w(x,t)\dd x=0,
\end{align*}
and hence,
\begin{align}\label{zero}
	\int_{\mt^n}\sigma(x,t)\dd x=0.
\end{align}
\eqref{poincarew} and \eqref{zero} indicate that the Poincaré-type inequalities also hold for $w$ and $\sigma$ in $\mt^n$.

Combining \eqref{d} and \eqref{Du} and making use of the definitions of $w$ and $\sigma$ in \eqref{definitionsigma}, we can establish the following system
\begin{align}\label{wsigma}
	 \renewcommand{\arraystretch}{2}
	\left\{\begin{array}{l}
	\partial_t w+\dfrac{1}{\lambda+\mu}\left(|\tb|^2+1\right) w+\left(|\tb|^2+1\right) \operatorname{div} \sigma=F_1,\\
	\partial_t \sigma-(\lambda+\mu) \Delta \sigma-\dfrac{1}{\lambda+\mu}\left(|\tb|^2+1\right) \mathbb{D} u=F_2,
\end{array}\right.
\end{align}
where
\begin{align}\label{F1}
	&F_1=\tb \cdot \nabla u\cdot \tb+f_1+f_3 \cdot \tb,\\\label{F2}
	&F_2=-\dfrac{1}{\lambda+\mu} \Delta^{-1} \nabla(\tb \cdot \nabla u \cdot \tb)+\mathbb{D} f_4-\dfrac{1}{\lambda+\mu} \Delta^{-1} \nabla\left(f_1+f_3 \cdot \tb\right).
\end{align}
\begin{lemma}\label{lem2}
	Let
\begin{align*}
	\bar{w}=\sqrt{\dfrac{3\left(\lambda+\mu\right)}{8C_1}}w,
\end{align*}
it holds that
\begin{align}\nonumber
	&\dfrac{\dd}{\dd t}\left(\| \sigma\|_{H^\iota}^2+\| \bar{w}\|_{H^\iota}^2\right)+(\lambda+\mu)\|\nabla   \sigma\|_{  H^\iota}^2+\dfrac{1}{2(\lambda+\mu)}\|\bar{w}\|_{  H^\iota}^2\\\nonumber
	\leq& C_4\|\nabla  u\|_{  H^\iota}+C\|\nabla  (a,u,B)\|_{  H^\iota}^2\|(a,u,B)\|_{L^\infty}^2\\\nonumber
	&+C \left(\|\nabla  u\|_{  H^\iota}^2\|u\|_{L^\infty}^2+\|a\|_{L^\infty}^2\|\nabla  u\|_{  H^\iota}^2+\|a\|_{H^{\iota-1}}^2\|\nabla^2 u\|_{L^\infty}^2\right)\\\label{4.0}
	&+C\|  (a,u,B)\|_{  H^\iota}^2\left(\|(a,u,B)\|_{W^{1,\infty}}^2+\|(a,B)\|_{W^{1,\infty}}^4\right).
\end{align}
for any $\iota\geq 1$, and
\begin{align*}\nonumber
	&\dfrac{\dd}{\dd t}\left(\|  \sigma\|_{L^2}^2+\|  \bar{w}\|_{L^2}^2\right)+(\lambda+\mu)\|\nabla   \sigma\|_{L^2}^2+\dfrac{1}{2(\lambda+\mu)}\|  \bar{w}\|_{L^2}^2\\\nonumber
	\leq& C_4\|\nabla  u\|_{L^2}+C\|\nabla  (a,u,B)\|_{L^2}^2\|(a,u,B)\|_{L^\infty}^2\\\nonumber
	&+C \left(\|\nabla  u\|_{L^2}^2\|u\|_{L^\infty}^2+\|a\|_{L^\infty}^2\|\nabla  u\|_{L^2}^2+\|a\|_{L^2}^2\|\nabla^2 u\|_{L^\infty}^2\right)\\
	&+C\|  (a,u,B)\|_{L^2}^2\left(\|(a,u,B)\|_{W^{1,\infty}}^2+\|(a,B)\|_{W^{1,\infty}}^4\right),
\end{align*}
where $C_4>0$ is an absolute constant.
\end{lemma}
\begin{proof}
{\bf When $\iota \geq1 $}, applying the operator $\Lambda^\iota$ to the equation \eqref{wsigma}$_1$ and multiplying the resultant by $\Lambda^\iota w$ yields
\begin{align} \nonumber
		&\dfrac{1}{2}\dfrac{\dd}{\dd t}\|\Li w\|_{L^2}^2+\dfrac{1}{\lambda+\mu}\left(|\tb|^2+1\right)\|\Li w\|_{L^2}^2\\\label{4.1}
	=&-(|\tb|^2+1)\int_{\mathbb{T}^n}\Li \san\sigma\cdot \Li w\dd x+\int_{\mathbb{T}^n}\Li F_1\cdot \Li w\dd x
\end{align}
To estimate the last term of \eqref{4.1}, by recalling the definition of $F_1$, one has
\begin{align}\nonumber
	&\int_{\mathbb{T}^n}\Li F_1\cdot \Li w\dd x\\
	\leq& \|\Li w\|_{L^2}\|\Li F_1\|_{L^2}\nonumber\\
	\leq& C \|\Li w\|_{L^2}\left(\|\nabla\Lambda^{\iota} u\|_{L^2}+\left\|\Li\left(f_1+f_3\cdot\tb\right)\right\|_{L^2} \right).\label{F11}
\end{align}
Noticing that $f_1$ and $f_3\cdot\tb$ can be rewritten as
\begin{align*}
	f_1&=-u\cdot \nabla a-a\nabla\cdot u=-\nabla \cdot(au),\\
	f_3\cdot \tb&=\left(-u\cdot\nabla B+B\cdot\nabla u-B\nabla\cdot u\right)\cdot\tb\\
	&=-u\cdot\nabla \left(B\cdot \tb\right)+\left(B\cdot\nabla u\right)\cdot\tb -\left(B\cdot \tb\right)\nabla\cdot u\\
	&=-\nabla\cdot\left(u(B\cdot\tb)\right)+B\cdot\nabla(u\cdot\tb)\\
	&=-\nabla\cdot\left(u(B\cdot\tb)\right)+\nabla\cdot\left(B(u\cdot \tb)\right),
\end{align*}
we can estimate the last term in \ref{F11} as follows
\begin{align}\label{4.21}
	\left\|\Li\left(f_1+f_3\cdot\tb\right)\right\|_{L^2}\leq |\tb|\|\nabla\Lambda^{\iota}(a,u,B)\|_{L^2}\|(a,u,B)\|_{L^\infty}.
\end{align}
Combing \eqref{4.1} and \eqref{4.21}, it follows that
	\begin{align*}
	&\dfrac{1}{2}\dfrac{\dd}{\dd t}\|\Li w\|_{L^2}^2+\dfrac{1}{\lambda+\mu}\left(|\tb|^2+1\right)\|\Li w\|_{L^2}^2\\
	\leq& C\|\nabla\Lambda^{\iota} \sigma\|_{L^2}\|\Li w\|_{L^2}+C \|\Li w\|_{L^2}\left(\|\nabla \Lambda^{\iota} u\|_{L^2}+\|\nabla\Lambda^{\iota}(a,u,B)\|_{L^2}\|(a,u,B)\|_{L^\infty} \right)\\
	\leq& \dfrac{1}{8(\lambda+\mu)}\|\Li w\|_{L^2}^2+C_1\left(\|\nabla \Lambda^{\iota} u\|_{L^2}^2+\|\nabla\Lambda^{\iota} \sigma\|_{L^2}^2+\|\nabla\Lambda^{\iota}(a,u,B)\|_{L^2}^2\|(a,u,B)\|_{L^\infty}^2  \right),
\end{align*}
which implies, after subtracting the first term on the right side, that
\begin{align}\nonumber
	&\dfrac{1}{2}\dfrac{\dd}{\dd t}\|\Li w\|_{L^2}^2+\dfrac{1}{2(\lambda+\mu)}\left(|\tb|^2+1\right)\|\Li w\|_{L^2}^2\\\label{hign w}
	\leq& C_1\left(\|\nabla \Lambda^{\iota} u\|_{L^2}^2+\|\nabla\Lambda^{\iota} \sigma\|_{L^2}^2+\|\nabla\Lambda^{\iota}(a,u,B)\|_{L^2}^2\|(a,u,B)\|_{L^\infty}^2  \right).
\end{align}

Again, applying the operator $\Lambda^\iota$ to the equation \eqref{wsigma}$_2$ and taking inner product of the resultant with $\Li \sigma$ yields
\begin{align}\nonumber
	&\dfrac{1}{2}\dfrac{\dd}{\dd t}\|\Li \sigma\|_{L^2}^2+(\lambda+\mu)\|\nabla\Lambda^{\iota} \sigma\|_{L^2}^2\\\label{4.2}
	=&\dfrac{(|\tb|^2+1)}{\lambda+\mu}\int_{\mt^n}\Li \mathbb{D}u\cdot\Li \sigma\dd x+\int_{\mt^n}\Li F_2\cdot\Li \sigma\dd x.
\end{align}
Recalling the definition of $F_2$ in \eqref{F2} and integration by parts, we have
\begin{align}\nonumber
	&\left|	\int_{\mt^n}\Li F_2\cdot\Li \sigma\dd x \right|\\\label{4.4}
	\leq& C\left(\|\nabla\Lambda^{\iota}u\|_{L^2}+\left\|\Lambda^{\iota}\left(f_1+f_3\cdot\tb\right)\right\|_{L^2}\right)\|\Lambda^{\iota-1}\sigma\|_{L^2}+\left|\int_{\mt^n}\Li\mathbb{D}f_4\cdot\Li \sigma\dd x\right|.
\end{align}
To estimate the last term in \eqref{4.4}, we first recall the following definition of $f_4$ in \eqref{f41}
as
	\begin{align*}
	\begin{aligned}
		f_4 \stackrel{\text { def }}{=} & - u \cdot \nabla  u+ B \cdot \nabla  B+ 	\nabla( B \cdot  B)+k(a) \nabla a -I(a)\left(\mu \Delta  u+\lambda \nabla \operatorname{div}  u\right)\\
		& -I(a)\left(\tb \cdot \nabla  B+ B \cdot \nabla  B-\nabla(\tb \cdot B) -\nabla( B \cdot  B)\right),
	\end{aligned}
\end{align*}
and then inserting it into the last term of \eqref{4.4} gives
\begin{align}\label{4.6}
	\left|\int_{\mt^n}\Li\mathbb{D}f_4\cdot\Li \sigma\dd x\right|\leq Z_1+Z_2+Z_3+Z_4+Z_5+Z_6,
\end{align}
where
\begin{align*}
	&Z_1=	\left|\int_{\mt^n}\Li\mathbb{D}\left(- u \cdot \nabla  u\right)\cdot\Li \sigma\dd x\right|,\\
	&Z_2=\left|\int_{\mt^n}\Li\mathbb{D}\left(B \cdot \nabla  B+ 	\nabla( B \cdot  B)\right)\cdot\Li \sigma\dd x\right|,\\
	&Z_3=\left|\int_{\mt^n}\Li\mathbb{D}\left(k(a) \nabla a \right)\cdot\Li \sigma\dd x\right|,\\
	& Z_4=\left|\int_{\mt^n}\Li\mathbb{D}\left(-I(a)\left(\mu \Delta  u+\lambda \nabla \operatorname{div}  u\right)\right)\cdot\Li \sigma\dd x\right|,\\
	&Z_5=\left|\int_{\mt^n}\Li\mathbb{D}\left(-I(a)\left(\tb \cdot \nabla  B-\nabla(\tb \cdot B) \right)\right)\cdot\Li \sigma\dd x\right|,\\
	&Z_6=\left|\int_{\mt^n}\Li\mathbb{D}\left(-I(a)\left(B \cdot \nabla  B-\nabla(B \cdot B) \right)\right)\cdot\Li \sigma\dd x\right|.
\end{align*}
For the first and second terms, applying Lemma \ref{commutator} and integration by parts, it follows that
\begin{align*}
	Z_1=	\left|\int_{\mt^n}\Li\mathbb{D}\left(- u \cdot \nabla  u\right)\cdot\Li \sigma\dd x\right|\leq C\|\Lambda^{\iota}\sigma\|_{L^2}\left(\|\Lambda^{\iota}u\|_{L^2}\|\nabla u\|_{L^\infty}+\|\nabla\Li u\|_{L^2}\|u\|_{L^\infty}\right),
\end{align*}
and
\begin{align*}
	Z_2=\left|\int_{\mt^n}\Li\mathbb{D}\left(B \cdot \nabla  B+ 	\nabla( B \cdot  B)\right)\cdot\Li \sigma\dd x\right|\leq C\|\nabla\Lambda^{\iota}\sigma\|_{L^2}\|\Li B\|_{L^2}\|B\|_{L^\infty}.
\end{align*}
Similar as before, one can use \eqref{Ia}-\eqref{Ia1} to estimate the terms $Z_3$--$Z_6$ as follows,
\begin{align*}
	Z_3&=\left|\int_{\mt^n}\Li\mathbb{D}\left(k(a) \nabla a \right)\cdot\Li \sigma\dd x\right|\\
	&\leq C\|\nabla\Lambda^{\iota}\sigma\|_{L^2}\left(\|\Lambda^{\iota-1}k(a)\|_{L^2}\|\nabla a\|_{L^\infty}+\|\Li a\|_{L^2}\|k(a)\|_{L^\infty}\right)\\
	&\leq C\|\nabla\Lambda^{\iota}\sigma\|_{L^2}\left(\|\Lambda^{\iota-1}a\|_{L^2}\|\nabla a\|_{L^\infty}+\|\Li a\|_{L^2}\|a\|_{L^\infty}\right),
\end{align*}
\begin{align*}
	Z_4&=\left|\int_{\mt^n}\Li\mathbb{D}\left(-I(a)\left(\mu \Delta  u+\lambda \nabla \operatorname{div}  u\right)\right)\cdot\Li \sigma\dd x\right|\\
	&\leq C\|\nabla\Lambda^{\iota}\sigma\|_{L^2}\left(\|I(a)\|_{L^\infty}\|\nabla\Lambda^{\iota}u\|_{L^2}+\|\Lambda^{\iota-1}I(a)\|_{L^2}\|\nabla^2 u\|_{L^\infty}\right)\\
	&\leq C\|\nabla\Lambda^{\iota}\sigma\|_{L^2}\left(\|a\|_{L^\infty}\|\nabla\Lambda^{\iota}u\|_{L^2}+\|\Lambda^{\iota-1}a\|_{L^2}\|\nabla^2 u\|_{L^\infty}\right),
\end{align*}
\begin{align*}
	Z_5&=\left|\int_{\mt^n}\Li\mathbb{D}\left(-I(a)\left(\tb \cdot \nabla  B-\nabla(\tb \cdot B) \right)\right)\cdot\Li \sigma\dd x\right|\\
	&\leq C\|\nabla\Lambda^{\iota}\sigma\|_{L^2}\left(\|I(a)\|_{L^\infty}\|\Lambda^{\iota}B\|_{L^2}+\|\Lambda^{\iota-1}I(a)\|_{L^2}\|\Lambda B\|_{L^\infty}\right)\\
	&\leq C\|\nabla\Lambda^{\iota}\sigma\|_{L^2}\left(\|a\|_{L^\infty}\|\Lambda^{\iota}B\|_{L^2}+\|\Lambda^{\iota-1}a\|_{L^2}\|\nabla B\|_{L^\infty}\right),
\end{align*}
and
\begin{align*}
	Z_6&=\left|\int_{\mt^n}\Li\mathbb{D}\left(-I(a)\left(B \cdot \nabla  B-\nabla(B \cdot B) \right)\right)\cdot\Li \sigma\dd x\right|\\
	&\leq C\|\nabla\Lambda^{\iota}\sigma\|_{L^2}\|I(a)\|_{L^\infty}\|\Lambda^{\iota-1}\left(B \cdot \nabla  B-\nabla(B \cdot B)\right)\|_{L^2}\\
	&~~~~~~+C\|\nabla\Lambda^{\iota}\sigma\|_{L^2}\|\Lambda^{\iota-1}I(a)\|_{L^2}\|B \cdot \nabla  B-\nabla(B \cdot B)\|_{L^\infty}\\
	&\leq C\|\nabla\Lambda^{\iota}\sigma\|_{L^2}\left(\|a\|_{L^\infty}\|\Lambda^{\iota}B\|_{L^2}\|B\|_{L^\infty}+\|a\|_{L^\infty}\|\Lambda^{\iota-1}B\|_{L^2}\|\nabla B\|_{L^\infty}\right)\\
	&~~~~~~+C\|\nabla\Lambda^{\iota}\sigma\|_{L^2}\|\Lambda^{\iota-1}a\|_{L^2}\|\nabla B\|_{L^\infty}\|B\|_{L^\infty}.
\end{align*}
Combining all the estimates above and \eqref{4.6}, we can summary them as
\begin{align*}
&	\left|\int_{\mt^n}\Li\mathbb{D}f_4\cdot\Li \sigma\dd x\right|\\
\leq& C\|\nabla\Lambda^{\iota}\sigma\|_{L^2}\|\Lambda^{\iota}(a,u,B)\|_{L^2}\left(\|(a,u,B)\|_{W^{1,\infty}}+\|(a,B)\|_{W^{1,\infty}}^2\right)\\
	&+C\|\nabla\Lambda^{\iota}\sigma\|_{L^2}\left(\|\nabla\Li u\|_{L^2}\|u\|_{L^\infty}+\|a\|_{L^\infty}\|\nabla\Lambda^{\iota}u\|_{L^2}+\|\Lambda^{\iota-1}a\|_{L^2}\|\nabla^2 u\|_{L^\infty}\right),
\end{align*}
which implies, after combining with \eqref{4.21}, that
\begin{align}\label{F2}
	&\left|	\int_{\mt^n}\Li F_2\cdot\Li \sigma\dd x \right|\nonumber\\
	\leq& C\left(\|\nabla\Lambda^{\iota}u\|_{L^2}+\left\|\Lambda^{\iota}\left(f_1+f_3\cdot\tb\right)\right\|_{L^2}\right)\|\Lambda^{\iota-1}\sigma\|_{L^2}+\left|\int_{\mt^n}\Li\mathbb{D}f_4\cdot\Li \sigma\dd x\right|\nonumber\\
	\leq& C \left(\|\nabla\Lambda^{\iota}u\|_{L^2}+|\tb|\|\nabla\Lambda^{\iota}(a,u,B)\|_{L^2}\|(a,u,B)\|_{L^\infty}\right)\|\nabla\Lambda^{\iota}\sigma\|_{L^2}\nonumber\\
	&+C\|\nabla\Lambda^{\iota}\sigma\|_{L^2}\left(\|\nabla\Li u\|_{L^2}\|u\|_{L^\infty}+\|a\|_{L^\infty}\|\nabla\Lambda^{\iota}u\|_{L^2}+\|\Lambda^{\iota-1}a\|_{L^2}\|\nabla^2 u\|_{L^\infty}\right)\nonumber\\
	&+C\|\nabla\Lambda^{\iota}\sigma\|_{L^2}\|\Lambda^{\iota}(a,u,B)\|_{L^2}\left(\|(a,u,B)\|_{W^{1,\infty}}+\|(a,B)\|_{W^{1,\infty}}^2\right)\nonumber\\
	\leq& \dfrac{\lambda+\mu}{16}\|\nabla\Lambda^{\iota}\sigma\|_{L^2}^2+C_2\|\nabla\Lambda^{\iota}u\|_{L^2}+C\|\nabla\Lambda^{\iota}(a,u,B)\|_{L^2}^2\|(a,u,B)\|_{L^\infty}^2\nonumber\\
	&+C \left(\|\nabla\Li u\|_{L^2}^2\|u\|_{L^\infty}^2+\|a\|_{L^\infty}^2\|\nabla\Lambda^{\iota}u\|_{L^2}^2+\|\Lambda^{\iota-1}a\|_{L^2}^2\|\nabla^2 u\|_{L^\infty}^2\right)\nonumber\\
		&+C\|\Lambda^{\iota}(a,u,B)\|_{L^2}^2\left(\|(a,u,B)\|_{W^{1,\infty}}^2+\|(a,B)\|_{W^{1,\infty}}^4\right).
\end{align}
Inserting \eqref{F2} into \eqref{4.2}, we obtain
\begin{align*}
		&\dfrac{1}{2}\dfrac{\dd}{\dd t}\|\Li \sigma\|_{L^2}^2+(\lambda+\mu)\|\nabla\Lambda^{\iota} \sigma\|_{L^2}^2\\
	=&\dfrac{(|\tb|^2+1)}{\lambda+\mu}\int_{\mt^n}\Li \mathbb{D}u\cdot\Li \sigma\dd x+\int_{\mt^n}\Li F_2\cdot\Li \sigma\dd x\\
	\leq& C\|\nabla\Lambda^{\iota}u\|_{L^2}\|\Lambda^{\iota-1}\sigma\|_{L^2}+	\left|	\int_{\mt^n}\Li F_2\cdot\Li \sigma\dd x \right|\\
	\leq& \dfrac{\lambda+\mu}{8}\|\nabla\Lambda^{\iota}\sigma\|_{L^2}^2+C_3\|\nabla\Lambda^{\iota}u\|_{L^2}^2+C\|\nabla\Lambda^{\iota}(a,u,B)\|_{L^2}^2\|(a,u,B)\|_{L^\infty}^2\\
	&+ C \left(\|\nabla\Li u\|_{L^2}^2\|u\|_{L^\infty}^2+\|a\|_{L^\infty}^2\|\nabla\Lambda^{\iota}u\|_{L^2}^2+\|\Lambda^{\iota-1}a\|_{L^2}^2\|\nabla^2 u\|_{L^\infty}^2\right)\\
	&+C\|\Lambda^{\iota}(a,u,B)\|_{L^2}^2\left(\|(a,u,B)\|_{W^{1,\infty}}^2+\|(a,B)\|_{W^{1,\infty}}^4\right),
\end{align*}
i.e.,
\begin{align}\nonumber
	&\dfrac{1}{2}\dfrac{\dd}{\dd t}\|\Li \sigma\|_{L^2}^2+\dfrac{7}{8}(\lambda+\mu)\|\nabla\Lambda^{\iota} \sigma\|_{L^2}^2\\\nonumber
	\leq& C_3\|\nabla\Lambda^{\iota}u\|_{L^2}^2+C\|\nabla\Lambda^{\iota}(a,u,B)\|_{L^2}^2\|(a,u,B)\|_{L^\infty}^2\\\nonumber
	&+C \left(\|\nabla\Li u\|_{L^2}^2\|u\|_{L^\infty}^2+\|a\|_{L^\infty}^2\|\nabla\Lambda^{\iota}u\|_{L^2}^2+\|\Lambda^{\iota-1}a\|_{L^2}^2\|\nabla^2 u\|_{L^\infty}^2\right)\\\label{hign sigma}
	&+C\|\Lambda^{\iota}(a,u,B)\|_{L^2}^2\left(\|(a,u,B)\|_{W^{1,\infty}}^2+\|(a,B)\|_{W^{1,\infty}}^4\right).
\end{align}

Multiplying \eqref{hign w} with $\dfrac{3\left(\lambda+\mu\right)}{8C_1}$ and adding the resultant with \eqref{hign sigma}, it yields that
\begin{align}\nonumber
	&\dfrac{\dd}{\dd t}\left(\|\Li \sigma\|_{L^2}^2+\dfrac{3\left(\lambda+\mu\right)}{8C_1}\|\Li w\|_{L^2}^2\right)+(\lambda+\mu)\|\nabla\Lambda^{\iota} \sigma\|_{L^2}^2+\dfrac{1}{2(\lambda+\mu)}\dfrac{3\left(\lambda+\mu\right)}{8C_1}\|\Li w\|_{L^2}^2\\\nonumber
	\leq& C_4\|\nabla\Li u\|_{L^2}+C\|\nabla\Lambda^{\iota}(a,u,B)\|_{L^2}^2\|(a,u,B)\|_{L^\infty}^2\\\nonumber
	&+C \left(\|\nabla\Li u\|_{L^2}^2\|u\|_{L^\infty}^2+\|a\|_{L^\infty}^2\|\nabla\Lambda^{\iota}u\|_{L^2}^2+\|\Lambda^{\iota-1}a\|_{L^2}^2\|\nabla^2 u\|_{L^\infty}^2\right)\\\label{4.7}
	&+C\|\Lambda^{\iota}(a,u,B)\|_{L^2}^2\left(\|(a,u,B)\|_{W^{1,\infty}}^2+\|(a,B)\|_{W^{1,\infty}}^4\right),
\end{align}
which can be updated, by setting  $\bar{w}=\sqrt{\dfrac{3\left(\lambda+\mu\right)}{8C_1}}w$, that
\begin{align}\nonumber
	&\dfrac{\dd}{\dd t}\left(\|\Li \sigma\|_{L^2}^2+\|\Li \bar{w}\|_{L^2}^2\right)+(\lambda+\mu)\|\nabla\Lambda^{\iota} \sigma\|_{L^2}^2+\dfrac{1}{2(\lambda+\mu)}\|\Li \bar{w}\|_{L^2}^2\\\nonumber
	\leq& C_4\|\nabla\Li u\|_{L^2}+C\|\nabla\Lambda^{\iota}(a,u,B)\|_{L^2}^2\|(a,u,B)\|_{L^\infty}^2\\\nonumber
&+C \left(\|\nabla\Li u\|_{L^2}^2\|u\|_{L^\infty}^2+\|a\|_{L^\infty}^2\|\nabla\Lambda^{\iota}u\|_{L^2}^2+\|\Lambda^{\iota-1}a\|_{L^2}^2\|\nabla^2 u\|_{L^\infty}^2\right)\\\label{4.8}
&+C\|\Lambda^{\iota}(a,u,B)\|_{L^2}^2\left(\|(a,u,B)\|_{W^{1,\infty}}^2+\|(a,B)\|_{W^{1,\infty}}^4\right).
\end{align}

{\bf When $\iota = 0$}, it suffices to take $\iota=0$ in \eqref{4.8} and then we have
\begin{align}\nonumber
	&\dfrac{\dd}{\dd t}\left(\|  \sigma\|_{L^2}^2+\|  \bar{w}\|_{L^2}^2\right)+(\lambda+\mu)\|\nabla   \sigma\|_{L^2}^2+\dfrac{1}{2(\lambda+\mu)}\|  \bar{w}\|_{L^2}^2\\\nonumber
	\leq& C_4\|\nabla  u\|_{L^2}+C\|\nabla  (a,u,B)\|_{L^2}^2\|(a,u,B)\|_{L^\infty}^2\\\nonumber
	&+C \left(\|\nabla  u\|_{L^2}^2\|u\|_{L^\infty}^2+\|a\|_{L^\infty}^2\|\nabla  u\|_{L^2}^2+\|a\|_{L^2}^2\|\nabla^2 u\|_{L^\infty}^2\right)\\\label{4.9}
	&+C\|  (a,u,B)\|_{L^2}^2\left(\|(a,u,B)\|_{W^{1,\infty}}^2+\|(a,B)\|_{W^{1,\infty}}^4\right).
\end{align}
Summing up \eqref{4.8} and \eqref{4.9}, we can complete the proof.
\end{proof}

\section{Proof of Theorem \ref{thm}}
\subsection{The asymptotic behavior}

\,\,\,\,\,\,\,\,Before constructing the effective decay estimates, as an important preparation, we need to integrate all the a priori estimates established in section 3.

{\bf Integration of a priori estimates}: Initially, multiplying \eqref{4.0} with $\dfrac{c_0\mu}{16C_4}$ and adding the resultant with \eqref{3.00}, we obtain
\begin{align}\nonumber
	&\frac{1}{2}	\frac{\dd}{\dd t}\|     (a,u,B)\|_{H^\iota}^2+\frac{\dd}{\dd t}\left(\dfrac{c_0\mu}{16C_4}\|    (\sigma,\bar{w})\|_{H^\iota}^2\right)+\frac{1}{4}c_0\lambda \|   \san u\|_{H^\iota}^2+\frac{1}{8}c_0\mu \|\nabla    u\|_{H^\iota}^2\\\nonumber
	&+\dfrac{c_0\lambda}{16C_4}\left((\lambda+\mu)\|\nabla    \sigma\|_{H^\iota}^2+\dfrac{1}{2(\lambda+\mu)}\|    \bar{w}\|_{H^\iota}^2\right)\\\nonumber
\leq& C\|\nabla   (a,u,B)\|_{H^\iota}^2\|(a,u,B)\|_{L^\infty}^2 \\\nonumber
	&+C\|    (a,u,B)\|_{H^\iota}^2\left(\| (a,u,B)\|_{W^{1,\infty}}+\|(a,u,B)\|_{W^{1,\infty}}^2+\|(a,B)\|_{W^{1,\infty}}^4\right)\\\label{5.1}
	&+C \left(\|\nabla    u\|_{H^\iota}^2\|u\|_{L^\infty}^2+\|a\|_{L^\infty}^2\|\nabla   u\|_{H^\iota}^2+\|a\|_{H^{\iota-1}}^2\|\nabla^2 u\|_{L^\infty}^2\right),
\end{align}
which can be updated as,
\begin{align}\nonumber
	&	\frac{\dd}{\dd t}\|     (u,B,a,\bs,\bw)\|_{H^\iota}^2+\frac{1}{2}c_0\lambda \|\san u\|_{H^\iota}^2+\frac{1}{4}c_0\mu \|\nabla u\|_{H^\iota}^2+(\lambda+\mu)\|\nabla    \bs\|_{H^\iota}^2\\\nonumber
	&~~+\dfrac{1}{2(\lambda+\mu)}\|    \bw\|_{H^\iota}^2\\\nonumber
\leq& C\|\nabla   (a,u,B)\|_{H^\iota}^2\|(a,u,B)\|_{L^\infty}^2 \\\nonumber
	&+C\|    (a,u,B)\|_{H^\iota}^2\left(\| (a,u,B)\|_{W^{1,\infty}}+\|(a,u,B)\|_{W^{1,\infty}}^2+\|(a,B)\|_{W^{1,\infty}}^4\right)\\\label{5.2}
	&+C \left(\|\nabla    u\|_{H^\iota}^2\|u\|_{L^\infty}^2+\|a\|_{L^\infty}^2\|\nabla   u\|_{H^\iota}^2+\|a\|_{H^{\iota-1}}^2\|\nabla^2 u\|_{L^\infty}^2\right),
\end{align}
by setting
\begin{align*}
	\breve{w}=\sqrt{\dfrac{c_0\lambda}{8C_4}}\bar{w}=\sqrt{\dfrac{3\left(\lambda+\mu\right)c_0\lambda}{64C_1C_4}}w,\quad \breve{\sigma}=\sqrt{\dfrac{c_0\lambda}{64C_4}}\sigma,
\end{align*}

Subsequently, taking $\iota=\dfrac{s}{2}+1$ and multiplying \eqref{5.2} with
\begin{equation}\label{Adef}
A\triangleq \max\left\{\frac{16\mu^2+64|\tb|^2}{c_0\mu},|\tb|+1\right\},
\end{equation}
and adding the resultant with \eqref{3.90}, it yields that
\begin{align}\nonumber
	&	\frac{\dd}{\dd t}\left[A\|   (u,B,a,\bs,\bw)\|_{H^{\frac{s}{2}+1}}^2-\int_{\mathbb{T}^n}\left(\tb\cdot\nabla B\right)\cdot \Lambda^s \mathbb{P}u\dd x-\int_{\mathbb{T}^n}\left(\tb\cdot\nabla B\right)\cdot \mathbb{P}u\dd x\right]\\\nonumber
	&+A\left[\frac{1}{2}c_0\lambda \|  \san u\|_{H^{\frac{s}{2}+1}}^2+\frac{1}{4}c_0\mu \|\nabla  u\|_{H^{\frac{s}{2}+1}}^2+(\lambda+\mu)\|\nabla  \bs\|_{H^{\frac{s}{2}+1}}^2+\dfrac{1}{2(\lambda+\mu)}\|   \bw\|_{H^{\frac{s}{2}+1}}^2\right]\\\nonumber
	\leq&-c_1\left\|\tb\cdot \nabla B\right\|_{H^{\frac{s}{2}}}^2+\left(\dfrac{1}{2}\mu^2+2|\tb|^2\right)\|\nabla u\|_{H^{\frac{s}{2}+1}}^2\\\nonumber
&~+C\|\nabla   (a,u,B)\|_{H^{\frac{s}{2}+1}}^2\|(a,u,B)\|_{L^\infty}^2+C\|a\|_{H^{\frac{s}{2}}}^2\|\nabla^2 u\|_{L^\infty}^2\\\nonumber
	&~+C\|B\|_{H^{\frac{s}{2}+1}}\left(\|a\|_{L^\infty}\|\nabla u\|_{H^{\frac{s}{2}+1}}+\|a\|_{H^\frac{s}{2}}\|\nabla^2 u\|_{L^\infty}+\|\nabla  u\|_{H^{\frac{s}{2}+1}}\|u\|_{L^\infty}\right)\\\nonumber
	&~+C(a,u,B)\|_{H^{\frac{s}{2}+1}}^2\left(\| (a,u,B)\|_{W^{1,\infty}}+\|(a,u,B)\|_{W^{1,\infty}}^2+\|(a,B)\|_{W^{1,\infty}}^4\right).	
\end{align}
which further implies, after noticing $A\geq \dfrac{16\mu^2+64|\tb|^2}{c_0\mu}$ and using the fact
\begin{align}\label{5.5}
	\left(\dfrac{1}{2}\mu^2+2|\tb|^2\right)\|\nabla u\|_{H^{\frac{s}{2}+1}}^2\leq \dfrac{Ac_0\mu}{32}\|\nabla u\|_{H^{\frac{s}{2}+1}}^2,
\end{align}
that,
\begin{align}
	&	\frac{\dd}{\dd t}\left[A\|   (u,B,a,\bs,\bw)\|_{H^{\frac{s}{2}+1}}^2-\int_{\mathbb{T}^n}\left(\tb\cdot\nabla B\right)\cdot \Lambda^s \mathbb{P}u\dd x-\int_{\mathbb{T}^n}\left(\tb\cdot\nabla B\right)\cdot \mathbb{P}u\dd x\right]\nonumber\\
	&+A\left[\frac{1}{2}c_0\lambda \|  \san u\|_{H^{\frac{s}{2}+1}}^2+\frac{7}{32}c_0\mu \|\nabla  u\|_{H^{\frac{s}{2}+1}}^2+(\lambda+\mu)\|\nabla  \bs\|_{H^{\frac{s}{2}+1}}^2+\dfrac{1}{2(\lambda+\mu)}\|   \bw\|_{H^{\frac{s}{2}+1}}^2\right]\nonumber\\
	\leq&-c_1\left\|\tb\cdot \nabla B\right\|_{H^{\frac{s}{2}}}^2+N,\label{5.3}
\end{align}
where $N$ represents the following nonlinear terms
\begin{align}
  N&=C\|\nabla   (a,u,B)\|_{H^{\frac{s}{2}+1}}^2\|(a,u,B)\|_{L^\infty}^2+C\|a\|_{H^{\frac{s}{2}}}^2\|\nabla^2 u\|_{L^\infty}^2\nonumber\\
	&~~~+C\|B\|_{H^{\frac{s}{2}+1}}\left(\|a\|_{L^\infty}\|\nabla u\|_{H^{\frac{s}{2}+1}}+\|a\|_{H^\frac{s}{2}}\|\nabla^2 u\|_{L^\infty}+\|\nabla  u\|_{H^{\frac{s}{2}+1}}\|u\|_{L^\infty}\right)\nonumber\\
	&~~~+C(a,u,B)\|_{H^{\frac{s}{2}+1}}^2\left(\| (a,u,B)\|_{W^{1,\infty}}+\|(a,u,B)\|_{W^{1,\infty}}^2+\|(a,B)\|_{W^{1,\infty}}^4\right).\label{6.0}
\end{align}

At this stage, if one look closely at the structure of \eqref{5.3}, how to deal with the term $c_1\left\|\tb\cdot \nabla B\right\|_{H^{\frac{s}{2}}}^2$ plays a very key role in solving this problem.
{\bf In current work, we observe that the coupling and interaction between the term $c_1\left\|\tb\cdot \nabla B\right\|_{H^{\frac{s}{2}}}^2$ and half of $\dfrac{1}{2(\lambda+\mu)}\|   \bw\|_{H^{\frac{s}{2}+1}}^2$ helps greatly increase the decay rate of solution of system \eqref{equation}.} In what follows, we will show how to improve the decay rate in Lemma \ref{lem4} {\bf(the core part of this paper)} by taking full advantage of this point, together with the Diophantine condition satisfied by background magnetic filed, various interpolation inequalities and characteristics of the periodic domain, which is quite different with the method used by Wu and Zhai in \cite{zhai-3m-compressible-nonresistive}.

\begin{lemma}\label{lem4}
Let $m\geq 2r+\dfrac{n}{2}+3$, and
\begin{align*}
	\int_{\mt^n}a_0\dd x=\int_{\mt^n}B_0\dd x=0.
\end{align*}
For any $T>0$ and a sufficiently small $\delta>0$, assume $(a,u,B)$ satisfies
	\begin{align}\label{5.6}
		\sup_{t\in[0,T]}\left(\|a(t)\|_{  H^m}+\|u(t)\|_{  H^m}+\|B(t)\|_{  H^m}\right)\leq \delta,
	\end{align}
and Poincaré inequality
\begin{align}\label{poincare-uL2}
	\|u(t)\|_{L^2}^2\leq \|\nabla u(t)\|_{L^2}^2.
\end{align}
Then for any $0\leq t\leq T$ and $2r\leq s<n+2+2r$, the following decay estimates hold
\begin{align}\label{decay}
	\|a(t)\|_{H^{\frac{s}{2}+1}}+	\|u(t)\|_{H^{\frac{s}{2}+1}}+	\|B(t)\|_{H^{\frac{s}{2}+1}}\leq C\left(1+t\right)^{-\frac{m-\left(\frac{s}{2}+1\right)}{2(1+r)}}.
\end{align}
\end{lemma}
\begin{proof}
{\bf Retrieving of hidden dissipative terms I:} If analyzing the dissipative structure of \eqref{5.3}, we may discover that the second bracket in the left side includes almost all the dissipation terms of variables apart from $B$ and $a$. In the following part, we will work to recover them. To this end, we first invoke the Poincaré-type inequality involving the Diophantine condition Lemma \ref{Diophantinepoincare} to derive
\begin{align}\label{dissipationofB}
	-c_1 \left\|\tb\cdot \nabla B\right\|_{H^{\frac{s}{2}}}^2\leq -c_d\|B\|_{H^{\frac{s}{2}-r}}^2,
\end{align}
where the constant $c_d>0$ depends only on $d$. Next work is recovering the dissipative term of $a$. To solve it, we rewrite its definition as
\begin{align*}
	a=w-\tb\cdot B=\sqrt{\dfrac{64C_1C_4}{3\left(\lambda+\mu\right)c_0\lambda}}\breve{w}-\tb\cdot B,
\end{align*}
which also means
\begin{align}\nonumber
	\| a\|_{H^{\frac{s}{2}-r}}^2&\leq 	\| w\|_{H^{\frac{s}{2}-r}}^2+	|\tb|^2\| B\|_{H^{\frac{s}{2}-r}}^2\\\nonumber
	&\leq \dfrac{64C_1C_4}{3\left(\lambda+\mu\right)c_0\lambda}\| \breve{w}\|_{H^{\frac{s}{2}-r}}^2+	|\tb|^2\| B\|_{H^{\frac{s}{2}-r}}^2\\\label{dissipationofa1}
	&\leq \dfrac{64C_1C_4}{3\left(\lambda+\mu\right)c_0\lambda}\|\breve{w}\|_{H^{\frac{s}{2}+1}}^2+	|\tb|^2\| B\|_{H^{\frac{s}{2}-r}}^2.
\end{align}
From \eqref{dissipationofa1}, we can see that to restore the dissipation of $a$, it suffices to distribute part of the dissipation of $B$ and $\breve{w}$ and then integrate them. Following this idea,
we can recover it as
\begin{align}\nonumber
	&\dfrac{1}{2}\left(-c_d\| B\|_{H^{\frac{s}{2}-r}}^2-\dfrac{A}{2(\lambda+\mu)}\| \bw\|_{H^{\frac{s}{2}+1}}^2\right)\\\nonumber
	\leq& \min\left\{\dfrac{3(\lambda+\mu)c_0\lambda}{64C_1C_4}\times\left(-\dfrac{A}{4(\lambda+\mu)}\right),\dfrac{1}{|\tb|^2}\times \left(-\dfrac{c_d}{2}\right)\right\}\| a\|_{H^{\frac{s}{2}-r}}^2\\\nonumber
	=&\min\left\{-\dfrac{3Ac_0\lambda}{256C_1C_4},-\dfrac{c_d}{2|\tb|^2}\right\}\| a\|_{H^{\frac{s}{2}-r}}^2\\
   \triangleq& -c_a \| a\|_{H^{\frac{s}{2}-r}}^2.\label{5.4}
\end{align}

Now, we just need to collect \eqref{dissipationofB} and \eqref{5.4}, and then combine the resultant with \eqref{5.3}. For simplicity, we unify all the constants that appear in \eqref{5.3}, \eqref{dissipationofB} and \eqref{5.4} through introducing a new number
\begin{align}\label{varkappa}
	\varkappa=\min\left\{\dfrac{c_d}{2},c_a,\dfrac{7Ac_0\mu}{32},A(\lambda+\mu),\dfrac{A}{4(\lambda+\mu)}\right\}.
\end{align}
Thanks to \eqref{varkappa} and employing \eqref{5.4} and \eqref{5.5}, we can ultimately update \eqref{5.3} as
\begin{align}\nonumber
	&	\frac{\dd}{\dd t}\left[A\|   (u,B,a,\bs,\bw)\|_{H^{\frac{s}{2}+1}}^2-\int_{\mathbb{T}^n}\left(\tb\cdot\nabla B\right)\cdot \Lambda^s \mathbb{P}u\dd x-\int_{\mathbb{T}^n}\left(\tb\cdot\nabla B\right)\cdot \mathbb{P}u\dd x\right]\\\label{5.51}
	&~+\varkappa\left[\underbrace{\|B\|_{H^{\frac{s}{2}-r}}^2+\|a\|_{H^{\frac{s}{2}-r}}^2}_{III}+\|\nabla u\|_{H^{\frac{s}{2}+1}}^2+\|\nabla  \bs\|_{H^{\frac{s}{2}+1}}^2+\|   \bw\|_{H^{\frac{s}{2}+1}}^2\right]\leq N.
\end{align}
In this part, we only recover the dissipative effect of $B$ and $a$ in $H^{\frac{s}{2}-r}$, which is not enough to help us establishing desired decay rate. However, term $III$ together with the condition
\eqref{5.6} are able to control the nonlinear terms $N$. Next, we will first estimate the nonlinear terms and then come back to restore the dissipation of $B$ and $a$ sequentially.

{\bf Control of nonlinear terms:}
According to Lemma \ref{chazhi1}, for $\frac{n}{2}+1>\frac{s}{2}-r \Rightarrow s<n+2(1+r)$ and $g=a,B\,{\rm or}\,u$, there holds
\begin{align}\label{em1}
\|g\|_{L^{\infty}}\leq \|g\|_{H^{2r+\frac{n}{2}+3}}^{\frac{2r+n-s}{6r+6+n-s}}\|g\|_{H^{\frac{s}{2}-r}}^{\frac{4r+6}{6r+6+n-s}},\quad \|g\|_{H^{\frac{s}{2}+2}}\leq \|g\|_{H^{2r+\frac{n}{2}+3}}^{\frac{2r+4}{6r+6+n-s}}\|g\|_{H^{\frac{s}{2}-r}}^{\frac{4r+2+n-s}{6r+6+n-s}},
\end{align}
\begin{align}\label{em2}
\|g\|_{W^{1,\infty}}\leq \|g\|_{H^{2r+\frac{n}{2}+3}}^{\frac{2r+2+n-s}{6r+6+n-s}}\|g\|_{H^{\frac{s}{2}-r}}^{\frac{4r+4}{6r+6+n-s}},\quad	\|g\|_{H^{\frac{s}{2}+1}}\leq \|g\|_{H^{2r+\frac{n}{2}+3}}^{\frac{2r+2}{6r+6+n-s}}\|g\|_{H^{\frac{s}{2}-r}}^{\frac{4r+4+n-s}{6r+6+n-s}},
\end{align}
\begin{align}\label{em3}
\|g\|_{W^{2,\infty}}\leq \|g\|_{H^{2r+\frac{n}{2}+3}}^{\frac{2r+4+n-s}{6r+6+n-s}}\|g\|_{H^{\frac{s}{2}-r}}^{\frac{4r+2}{6r+6+n-s}},\quad	\|g\|_{H^{\frac{s}{2}}}\leq \|g\|_{H^{2r+\frac{n}{2}+3}}^{\frac{2r}{6r+6+n-s}}\|g\|_{H^{\frac{s}{2}-r}}^{\frac{4r+6+n-s}{6r+6+n-s}},
\end{align}
and therefore
\begin{align}\label{multi1}
\|g\|_{L^{\infty}}\|g\|_{H^{\frac{s}{2}+1}}\|g\|_{H^{\frac{s}{2}+2}}\leq\|g\|_{H^{2r+\frac{n}{2}+3}}\|g\|_{H^{\frac{s}{2}-r}}^2,
\end{align}
\begin{align}\label{multi2}
\|g\|_{W^{1,\infty}}	\|g\|_{H^{\frac{s}{2}+1}}^2\leq\|g\|_{H^{2r+\frac{n}{2}+3}}\|g\|_{H^{\frac{s}{2}-r}}^2,
\end{align}
\begin{align}\label{multi3}
\|g\|_{W^{2,\infty}}\|g\|_{H^{\frac{s}{2}}}	\|g\|_{H^{\frac{s}{2}+1}}\leq\|g\|_{H^{2r+\frac{n}{2}+3}}\|g\|_{H^{\frac{s}{2}-r}}^2.
\end{align}
Considering $m\geq 2r+\dfrac{n}{2}+3$, $r\geq n-1$, Sobolev embedding inequality and \eqref{5.6}, we have
\begin{align}\label{embedding}
		\|a,B,u\|_{  W^{1,\infty}}\leq C\|a,B,u\|_{  H^m}\leq C\delta,
	\end{align}
which implies, after using \eqref{multi1}--\eqref{multi3}, that
\begin{align}\label{right1}
&C\|B\|_{H^{\frac{s}{2}+1}}\left(\|a\|_{L^\infty}\|\nabla u\|_{H^{\frac{s}{2}+1}}+\|a\|_{H^\frac{s}{2}}\|\nabla^2 u\|_{L^\infty}+\|\nabla  u\|_{H^{\frac{s}{2}+1}}\|u\|_{L^\infty}\right)\nonumber\\
	&+C\left(\| (a,u,B)\|_{W^{1,\infty}}+\|(a,u,B)\|_{W^{1,\infty}}^2+\|(a,B)\|_{W^{1,\infty}}^4\right)\|  (a,u,B)\|_{H^{\frac{s}{2}+1}}^2\nonumber\\	
\leq& \left(C+C\delta+C{\delta}^3\right) \|(a,u,B)\|_{H^{2r+\frac{n}{2}+3}}\|(a,u,B)\|_{H^{\frac{s}{2}-r}}^2\nonumber\\
\leq& C\|(a,u,B)\|_{H^{2r+\frac{n}{2}+3}}\|(a,u,B)\|_{H^{\frac{s}{2}-r}}^2.
\end{align}
Similar as before, by utilizing Lemma \ref{chazhi1} again, for $\frac{n}{2}+1>\frac{s}{2}-r \Rightarrow s<n+2(1+r)$ and $h=a,B\,{\rm or}\,u$, we have
\begin{align}\label{em4}
\|h\|_{L^{\infty}}\leq \|h\|_{H^{r+\frac{n}{2}+2}}^{\frac{2r+n-s}{4r+4+n-s}}\|h\|_{H^{\frac{s}{2}-r}}^{\frac{2r+4}{4r+4+n-s}},\quad \|h\|_{H^{\frac{s}{2}+2}}\leq \|h\|_{H^{r+\frac{n}{2}+2}}^{\frac{2r+4}{4r+4+n-s}}\|h\|_{H^{\frac{s}{2}-r}}^{\frac{2r+n-s}{4r+4+n-s}},
\end{align}
\begin{align}\label{em5}
\|h\|_{W^{2,\infty}}\leq \|h\|_{H^{r+\frac{n}{2}+2}}^{\frac{2r+4+n-s}{4r+4+n-s}}\|h\|_{H^{\frac{s}{2}-r}}^{\frac{2r}{4r+4+n-s}},\quad	\|h\|_{H^{\frac{s}{2}}}\leq \|h\|_{H^{r+\frac{n}{2}+2}}^{\frac{2r}{4r+4+n-s}}\|h\|_{H^{\frac{s}{2}-r}}^{\frac{2r+4+n-s}{4r+4+n-s}},
\end{align}
which yields, after noticing $m\geq 2r+\dfrac{n}{2}+3$ and using \eqref{5.6}, that
\begin{align}\label{right2}
	&C\|a\|_{H^{\frac{s}{2}}}^2\|\nabla^2 u\|_{L^\infty}^2+C\|\nabla   (a,u,B)\|_{H^{\frac{s}{2}+1}}^2\|(a,u,B)\|_{L^\infty}^2\nonumber\\
\leq& C\|(a,u,B)\|_{H^{\frac{n}{2}+2+r}}^2 \|(a,u,B)\|_{H^{\frac{s}{2}-r}}^2\nonumber\\
	\leq& C \|(a,u,B)\|_{H^{2r+\frac{n}{2}+3}}^2\|(a,u,B)\|_{H^{\frac{s}{2}-r}}^2\nonumber\\
\leq& C\delta\|(a,u,B)\|_{H^{2r+\frac{n}{2}+3}}\|(a,u,B)\|_{H^{\frac{s}{2}-r}}^2.
\end{align}
Subsequently, by taking $\delta>0$ small enough such that
\begin{align*}
	C_5\|(a,u,B)\|_{H^{2r+\frac{n}{2}+3}}\leq C_5\delta\leq \dfrac{\varkappa}{2},
\end{align*}
we can summarize \eqref{right1} and \eqref{right2} as
\begin{align}\label{right3}
	N\leq& \left(C+C\delta\right)\|(a,u,B)\|_{H^{2r+\frac{n}{2}+3}}\|(a,u,B)\|_{H^{\frac{s}{2}-r}}^2\nonumber\\
\triangleq&C_5\|(a,u,B)\|_{H^{2r+\frac{n}{2}+3}}\|(a,u,B)\|_{H^{\frac{s}{2}-r}}^2\nonumber\\
\leq& C_5\delta \|(a,u,B)\|_{H^{\frac{s}{2}-r}}^2\leq \dfrac{\varkappa}{2}\|(a,u,B)\|_{H^{\frac{s}{2}-r}}^2.
\end{align}

With \eqref{right3} at hand, we can rewrite \eqref{5.51} as
\begin{align}\label{estabofdecay1}
	&	\frac{\dd}{\dd t}\left[A\|   (u,B,a,\bs,\bw)\|_{H^{\frac{s}{2}+1}}^2-\int_{\mathbb{T}^n}\left(\tb\cdot\nabla B\right)\cdot \Lambda^s \mathbb{P}u\dd x-\int_{\mathbb{T}^n}\left(\tb\cdot\nabla B\right)\cdot \mathbb{P}u\dd x\right]\nonumber\\
	\leq& -\dfrac{\varkappa}{2}\left[\|B\|_{H^{\frac{s}{2}-r}}^2+\|a\|_{H^{\frac{s}{2}-r}}^2+\|\nabla u\|_{H^{\frac{s}{2}+1}}^2+\|\nabla  \bs\|_{H^{\frac{s}{2}+1}}^2+\|   \bw\|_{H^{\frac{s}{2}+1}}^2\right].
\end{align}
\eqref{poincare-uL2} and \eqref{zero} lead to
\begin{align}\label{estabofdecay2}
	\|\nabla u\|_{H^{\frac{s}{2}+1}}^2+\|\nabla  \bs\|_{H^{\frac{s}{2}+1}}^2\geq \dfrac{1}{C_6}	\left(\|u\|_{H^{\frac{s}{2}+1}}^2+\|  \bs\|_{H^{\frac{s}{2}+1}}^2\right).
\end{align}
Set $\varsigma=\min\{\dfrac{\varkappa}{2},\dfrac{\varkappa}{2C_6}\}$ and apply \eqref{estabofdecay2}, \eqref{estabofdecay1} can be updated as
\begin{align}\nonumber
	&	\frac{\dd}{\dd t}\left[A\|   (u,B,a,\bs,\bw)\|_{H^{\frac{s}{2}+1}}^2-\int_{\mathbb{T}^n}\left(\tb\cdot\nabla B\right)\cdot \Lambda^s \mathbb{P}u\dd x-\int_{\mathbb{T}^n}\left(\tb\cdot\nabla B\right)\cdot \mathbb{P}u\dd x\right]\\\label{5.61}
	\leq& -\varsigma\left[\underbrace{\|B\|_{H^{\frac{s}{2}-r}}^2+\|a\|_{H^{\frac{s}{2}-r}}^2}_{III}+\|u\|_{H^{\frac{s}{2}+1}}^2+\| \bs\|_{H^{\frac{s}{2}+1}}^2+\|   \bw\|_{H^{\frac{s}{2}+1}}^2\right].
\end{align}

{\bf Retrieving of hidden dissipative terms II:}
As shown in \eqref{5.61}, it suffices to unearth more dissipative effect from the term $I$. For this purpose,
we introduce a large number $M$ to be chosen later such that $\frac{A}{M}\leq 1$ and employ \eqref{modezero-aB} and Plancherel's theorem to conclude
\begin{align}\notag
	&\frac{A}{M}\|\Lambda^{\frac{s}{2}+1} B\|_{L^2}^2-\|\Lambda^{\frac{s}{2}-r} B\|_{L^2}^2=\sum_{|k|\neq 0}\left(\frac{A}{M}|k|^{s+2}-|k|^{s-2r}\right)\left|\w{B}(k)\right|^2\label{AM}\\
	\leq& \frac{A}{M}\sum_{\frac{A}{M}\geq |k|^{-2r-2}}|k|^{s+2}\left|\w{B}(k)\right|^2\notag\\
	\leq& \left(\frac{A}{M}\right)^{1+\frac{m-\frac{s}{2}-1}{1+r}}\sum_{\frac{A}{M}\geq |k|^{-2r-2}}\left(\dfrac{M}{A}\right)^{\frac{m-\frac{s}{2}-1}{1+r}}\left(\frac{1}{|k|}\right)^{2m-s-2}|k|^{2m}\left|\w{B}(k)\right|^2\notag\\
	\leq& \left(\frac{A}{M}\right)^{1+\frac{m-\frac{s}{2}-1}{1+r}}\|\Lambda^{m}B\|_{L^2},
\end{align}
where we used $\frac{A}{M}\leq 1$, $|k|\geq 1$ and $m\geq \frac{s}{2}+1$ in the last step. To estimate the terms on the right side of \eqref{5.61}, we first make use of Poincaré-type inequalities of $a$ \eqref{apoincare} and $B$ \eqref{bpoincare} to derive
\begin{align}\label{Hspoincare-aB}
	\| a\|_{H^{\frac{s}{2}+1}}^2 \leq C_6	\|\Lambda^{\frac{s}{2}+1} a\|_{L^2}^2,~~\| B\|_{H^{\frac{s}{2}+1}}^2 \leq C_6	\|\Lambda^{\frac{s}{2}+1} B\|_{L^2}^2,
\end{align}
where $C_6\geq 1$ is an absolute constant , and then utilize \eqref{Hspoincare-aB} to obtain
\begin{align}\nonumber
	&-\varsigma\|B\|_{H^{\frac{s}{2}-r}}^2\leq 	-\varsigma\|\Lambda^{\frac{s}{2}-r}B\|_{L^2}^2\\\nonumber
=&	\varsigma\left(\frac{A}{M}\|\Lambda^{\frac{s}{2}+1} B\|_{L^2}^2-\|\Lambda^{\frac{s}{2}-r} B\|_{L^2}^2\right)-\dfrac{A\varsigma}{2M}\|\Lambda^{\frac{s}{2}+1} B\|_{L^2}^2\\\nonumber
	\leq& \varsigma\left(\frac{A}{M}\right)^{1+\frac{m-\frac{s}{2}-1}{1+r}}\|\Lambda^{m}B\|_{L^2}^2-\dfrac{\varsigma A}{M}\|\Lambda^{\frac{s}{2}+1} B\|_{L^2}^2\\\label{5.9}
	\leq& \varsigma\left(\frac{A}{M}\right)^{1+\frac{m-\frac{s}{2}-1}{1+r}}\|B\|_{H^m}^2-\dfrac{\varsigma A}{MC_6}\| B\|_{H^{\frac{s}{2}+1}}^2,
\end{align}
and
\begin{align}\nonumber
	&-\varsigma\|a\|_{H^{\frac{s}{2}-r}}^2\leq -\varsigma \|\Lambda^{\frac{s}{2}-r}a\|_{L^2}^2\\\nonumber
=&	{\varsigma}\left(\frac{A}{M}\|\Lambda^{\frac{s}{2}+1} a\|_{L^2}^2-\|\Lambda^{\frac{s}{2}-r} a\|_{L^2}^2\right)-\dfrac{A\varsigma}{M}\|\Lambda^{\frac{s}{2}+1} a\|_{L^2}^2\\\nonumber
	\leq& \varsigma\left(\frac{A}{M}\right)^{1+\frac{m-\frac{s}{2}-1}{1+r}}\|\Lambda^{m}a\|_{L^2}^2-\dfrac{\varsigma A}{M}\|\Lambda^{\frac{s}{2}+1} a\|_{L^2}^2\\\label{5.10}
	\leq& \varsigma\left(\frac{A}{M}\right)^{1+\frac{m-\frac{s}{2}-1}{1+r}}\|a\|_{H^m}^2-\dfrac{\varsigma A}{MC_6}\| a\|_{H^{\frac{s}{2}+1}}^2.
\end{align}
Adding \eqref{5.9} and \eqref{5.10} together, it holds that
\begin{align}
	&-\varsigma\left(\|B\|_{H^{\frac{s}{2}-r}}^2+\|a\|_{H^{\frac{s}{2}-r}}^2\right)\nonumber\\
\leq& -\dfrac{\varsigma A}{MC_6}\left(\| B\|_{H^{\frac{s}{2}+1}}^2+\| a\|_{H^{\frac{s}{2}+1}}^2\right)+
	\varsigma\left(\frac{A}{M}\right)^{1+\frac{m-\frac{s}{2}-1}{1+r}}\left(\|B\|_{H^m}^2+\|a\|_{H^m}^2\right),\label{6.01}
\end{align}
from which, we obtain the desired dissipative terms and \eqref{5.61} can be updated as
\begin{align}\nonumber
	&	\frac{\dd}{\dd t}\left[\underbrace{A\|   (u,B,a,\bs,\bw)\|_{H^{\frac{s}{2}+1}}^2-\int_{\mathbb{T}^n}\left(\tb\cdot\nabla B\right)\cdot \Lambda^s \mathbb{P}u\dd x-\int_{\mathbb{T}^n}\left(\tb\cdot\nabla B\right)\cdot \mathbb{P}u\dd x}_{I}\right]\\\label{6.02}
	\leq& -\dfrac{\varsigma}{MC_6}\left(\underbrace{A\|(u,B,a,\breve{\sigma},\breve{w})\|_{H^{\frac{s}{2}+1}}^2}_{IV}\right)+\varsigma\left(\frac{A}{M}\right)^{1+\frac{m-\frac{s}{2}-1}{1+r}}\left(\|a\|_{H^m}^2+\|B\|_{H^m}^2\right),
\end{align}
by using the fact $\dfrac{A}{MC_6}\leq 1$. We can see from \eqref{6.02} that the term $I$ and $IV$ don't match yet. In fact, through invoking Hölder inequality and Poincaré-type inequality of $B$ \eqref{bpoincare}, one can get
\begin{align}\label{5.7}
	&\left|\int_{\mathbb{T}^n}\left(\tb\cdot\nabla B\right)\cdot \Lambda^s \mathbb{P}u\dd x\right|+\left|	\int_{\mathbb{T}^n}\left(\tb\cdot\nabla B\right)\cdot \mathbb{P}u\dd x\right|\nonumber\\
\leq&|\tilde{b}| \|\Lambda^{\frac{s}{2}}B\|_{L^2}\|\Lambda^{\frac{s}{2}+1}u\|_{L^2}+ |\tilde{b}| \|\nabla B\|_{L^2}\|u\|_{L^2}\nonumber\\
\leq& \frac{|\tilde{b}|}{2}\|\Lambda^{\frac{s}{2}+1}u\|_{L^2}^2+\frac{|\tilde{b}|}{2}\|\Lambda^{\frac{s}{2}+1}B\|_{L^2}^2+\frac{|\tilde{b}|}{2}\|u\|_{L^2}^2+\frac{|\tilde{b}|}{2}\|\nabla B\|_{L^2}^2\nonumber\\
\leq& \frac{A}{2}\|\left(u,B\right)\|_{H^{\frac{s}{2}+1}}^2,
\end{align}
which means that the term $I$ and $IV$ are equivalent. Thanks to \eqref{6.02} and \eqref{5.7}, we finally obtain the following dissipative inequality
\begin{align}\nonumber
	&	\frac{\dd}{\dd t}\left[A\|   (u,B,a,\bs,\bw)\|_{H^{\frac{s}{2}+1}}^2-\int_{\mathbb{T}^n}\left(\tb\cdot\nabla B\right)\cdot \Lambda^s \mathbb{P}u\dd x-\int_{\mathbb{T}^n}\left(\tb\cdot\nabla B\right)\cdot \mathbb{P}u\dd x\right]\\\nonumber
	\leq& -\dfrac{\varsigma}{2MC_6}\left[A\|   (u,B,a,\bs,\bw)\|_{H^{\frac{s}{2}+1}}^2-\int_{\mathbb{T}^n}\left(\tb\cdot\nabla B\right)\cdot \Lambda^s \mathbb{P}u\dd x-\int_{\mathbb{T}^n}\left(\tb\cdot\nabla B\right)\cdot \mathbb{P}u\dd x\right]\\\label{5.11}
	&+\underbrace{\frac{C_7}{{M}^{1+\frac{m-\frac{s}{2}-1}{1+r}}}\left(\|a\|_{H^m}^2+\|B\|_{H^m}^2\right)}_{V},
\end{align}
where $C_7$ is an absolute constant depends only on ${|\tilde{b}|}$, $\mu$, $c_0$, $m$, $r$ and $\varsigma$.

{\bf Suitable Lyapunov functional}: Come to \eqref{5.11}, we have got enough dissipative effect but the price to pay is the appearance of forcing term $V$. Fortunately, it can be well controlled by constructing the suitable Lyapunov functional
\begin{align*}
	F(t)\triangleq A\|   (u,B,a,\bs,\bw)\|_{H^{\frac{s}{2}+1}}^2-\int_{\mathbb{T}^n}\left(\tb\cdot\nabla B\right)\cdot \Lambda^s \mathbb{P}u\dd x-\int_{\mathbb{T}^n}\left(\tb\cdot\nabla B\right)\cdot \mathbb{P}u\dd x,
\end{align*}
which satisfies
\begin{align*}
	F(t)\geq\frac{A}{2}\|   (u,B,a,\bs,\bw)\|_{H^{\frac{s}{2}+1}}^2,
\end{align*}
according to \eqref{5.7}. To establish corresponding decay estimate, we first take $M=A+\dfrac{\varsigma t}{2jC_6}$ and multiply $(A+\dfrac{\varsigma t}{2jC_6})^j$ on both sides of \eqref{5.11} to obtain
\begin{align}
	\frac{\dd}{\dd t}\left((A+\dfrac{\varsigma t}{2jC_6})^jF(t)\right)\leq C_7(A+\dfrac{\varsigma t}{2jC_6})^{j-1-\frac{m-\frac{s}{2}-1}{1+r}}\|(a,B)(t)\|_{H^m}^2,\label{4.5}
\end{align}
where $j>\dfrac{m-r-1}{1+r}\geq \dfrac{m-\frac{s}{2}-1}{1+r}$. Subsequently, integrating \eqref{4.5} from $0$ to $t$, it follows that
\begin{align}\label{finaldecay}
	(A+\dfrac{\varsigma t}{2jC_6})^jF(t)\leq A^jF(0)+C\int_{0}^t(A+\dfrac{\varsigma \tau}{2jC_6})^{j-1-\frac{m-\frac{s}{2}-1}{1+r}}\|(a,B)(t)\|_{H^m}^2\dd \tau,
\end{align}

Finally, multiplying $(1+t)^{\frac{m-\frac{s}{2}-1}{1+r}-j}$ on both sides of \eqref{finaldecay} and using \eqref{5.6}, we have
\begin{align*}
	&(1+t)^{\frac{m-\frac{s}{2}-1}{1+r}}F(t)\\
	\leq& C (1+t)^{\frac{m-\frac{s}{2}-1}{1+r}-j}A^j F(0)+C(1+t)^{\frac{m-\frac{s}{2}-1}{1+r}-j}\int_{0}^t(A+\dfrac{\varsigma \tau}{2jC_6})^{j-1-\frac{m-\frac{s}{2}-1}{1+r}}\|(a,B)(\tau)\|_{H^m}^2\dd \tau\\
	\leq& C+C(1+t)^{\frac{m-\frac{s}{2}-1}{1+r}-j+j-\frac{m-\frac{s}{2}-1}{1+r}}\sup_{0\leq\tau\leq t}\|(a,B)(\tau)\|_{H^m}^2\leq C,
\end{align*}
and therefore
\begin{align}
	\|   (u,B,a,\bs,\bw)\|_{H^{\frac{s}{2}+1}}^2\leq C(1+t)^{-\frac{m-\frac{s}{2}-1}{1+r}},\label{Hsdecay}
\end{align}
which completes all the proof.
\end{proof}

\subsection{Proof of Theorem \ref{thm}}

\begin{proof}[Proof of Theorem \ref{thm}]
As mentioned in  Subsection 3.1, in what follows, it suffices to establish the global a priori estimates of solutions in $H^m$ and then use the bootstrapping argument to prove our main result. For this purpose, firstly, by utilizing Lemma \ref{lem4} and setting $s$ to be equal to $2r$, we can obtain the decay estimates
	\begin{align}\label{lower}
		\|   (a,u,B)\|_{H^{r+1}}\leq C(1+t)^{-\frac{m-r-1}{2+2r}},
	\end{align}
	which also infers, after using the imbedding inequalities and the condition $r+1>n\geq \frac{n}{2}+1$, that
	\begin{align}
		\|a\|_{W^{1,\infty}}+	\|u\|_{W^{1,\infty}}+\| B\|_{W^{1,\infty}}\leq C	\|(a,u,B)(t)\|_{H^{r+1}}\leq C(1+t)^{-\frac{m-r-1}{2+2r}}.\label{linfinity}
	\end{align}
	Combing \eqref{linfinity} and Lemma \ref{hign energy} with $\iota=m$, it follows that
	\begin{align}\nonumber
		&\frac{1}{2}	\frac{\dd}{\dd t}\|    (a,u,B)\|_{  H^m}^2+\frac{1}{4}c_0\lambda \|  \san u\|_{  H^m}^2+\frac{3}{16}c_0\mu \|\nabla   u\|_{  H^m}^2 \\\label{6.1}
		\leq& C\|   (a,u,B)\|_{  H^m}^2\left(\| (a,u,B)\|_{W^{1,\infty}}+\|(a,u,B)\|_{W^{1,\infty}}^2+\|(a,B)\|_{W^{1,\infty}}^4\right).
	\end{align}
	
	Invoking the assumption $m>3r+3\geq 2r+\frac{n}{2}+3$, it is clear that
    \begin{equation}\label{rangeofm}
     \frac{m-r-1}{2+2r}>1,
    \end{equation}
which together with \eqref{linfinity} imply
	\begin{align}\label{useful bound}
	\int_0^\infty\left(\| (a,u,B)(\tau)\|_{W^{1,\infty}}+\|(a,u,B)(\tau)\|_{W^{1,\infty}}^2+\|(a,B)(\tau)\|_{W^{1,\infty}}^4\right)\dd \tau\leq C.
	\end{align}
and further
\begin{align}
		\|(a,u,B)(t)\|_{H^m}^2\leq  C\varepsilon^2,\quad t\in[0,T].\label{Hm}
	\end{align}
by using Gronwall's inequality for \eqref{6.1}. Thanks to \eqref{Hm}, one can choose $\varepsilon$ be sufficiently small such that $\sqrt{C}\varepsilon < \frac{1}{2}\delta$. Then, according to the bootstrapping argument, we can conclude that the unique local solution can be extended to a global one, and satisfies the following estimate
	\begin{align}\label{hign}
			\|(a,u,B)(t)\|_{H^m}^2\leq  C\varepsilon,\quad t\in[0,\infty).
	\end{align}

	In the last part, we establish the decay estimates of solution in $H^\alpha$ for any $r+1\leq \alpha\leq m$. As a matter of fact, it suffices to apply interpolation inequality,
    and then we can obtain from \eqref{lower} and \eqref{hign}, that
	\begin{align}\label{lastdecay}
		&\|a(t)\|_{H^\alpha}+\|u(t)\|_{H^\alpha}+\|B(t)\|_{H^\alpha}\nonumber\\
\leq& C\|a(t)\|_{H^{r+1}}^{\frac{m-\alpha}{m-r-1}}\|a(t)\|_{H^m}^{\frac{\alpha-r-1}{m-r-1}}+C\|u(t)\|_{H^{r+1}}^{\frac{m-\alpha}{m-r-1}}\|u(t)\|_{H^m}^{\frac{\alpha-r-1}{m-r-1}}
		+C\|B(t)\|_{H^{r+1}}^{\frac{m-\alpha}{m-r-1}}\|B(t)\|_{H^m}^{\frac{\alpha-r-1}{m-r-1}}\nonumber\\
		\leq& C(1+t)^{-\frac{m-\alpha}{2+2r}},
	\end{align}
	which finish the proof of Theorem \ref{thm}.
\end{proof}


\section*{Ackonwledgments}
\,\,\,\,\,\,\,\,Quansen Jiu was partially supported by National Natural Science Foundation of China under grants (No.\,\,11931010, No.\,\,12061003). Jitao Liu was partially supported by National Natural Science Foundation of China under grants (No.\,\,11801018, No.\,\,12061003), Beijing Natural Science Foundation under grant (No.\,\,1192001) and Beijing University of Technology under grant (No.\,\,006000514123513).

\end{document}